\documentclass{amsart}

\usepackage{amsmath,amssymb,amsthm,amsfonts,mathrsfs}
\usepackage{graphicx,epsfig,xcolor,float}
\usepackage[utf8]{inputenc} 
\usepackage{hyperref}
\usepackage{amssymb}
\usepackage{caption}
\usepackage{subcaption}

\calclayout

\def\R{\mathbb R}
\def\T{\mathbb T}
\def\Z{\mathbb Z}

\def\S{\mathbb S}

\numberwithin{equation}{section}

\newtheorem{theorem}{Theorem}

\newtheorem{lemma}[theorem]{Lemma}
\newtheorem{proposition}[theorem]{Proposition}

\newtheorem{definition}[theorem]{Definition}
\newtheorem{remark}[theorem]{Remark}

\numberwithin{theorem}{section}

\makeindex

\begin{document}

\title[Almost prescribing scalar curvature by mixed convex integration]{Almost prescribing scalar curvature\\ by mixed convex integration}


\author{F.~Aliouane}
\address{Université Mohammed Seddik Benyahia, LAOTI, Département de mathématiques, faculté des Sciences Exactes et Informatique, BP 98, Ouled Aissa, 18000, Jijel, Algeria}\email{f{\_}aliouane@univ-jijel.dz}

\author{L.~Rifford}
\address{Universit\'e C\^ote d'Azur, CNRS, Labo. J.-A. Dieudonn\'e,  UMR CNRS 7351, Parc Valrose, 06108 Nice Cedex 02, France \& AIMS Senegal, Km 2, Route de Joal, Mbour, Senegal} \email{ludovic.rifford@math.cnrs.fr}

\author{M.~Theillière}
\address{Université de Rennes, CNRS, IRMAR - UMR 6625, F-35000 Rennes, France}\email{melanie.theilliere@univ-rennes.fr}

\date{}



\maketitle

\begin{abstract}
We introduce a method of mixed convex integration and demonstrate its suitability for solving a particular class of semilinear second-order partial differential relations. As an application, we provide a new proof of a result on scalar curvature originally established by Lohkamp.
\end{abstract}


\section{Introduction}
Let $\pi: E \rightarrow M$ be a smooth fiber bundle, where $E$ is the total space and $M$ is the base manifold. A section $s:M\rightarrow E$ corresponds to a smooth map such that $\pi(s(x))=x$ for all $x\in M$. The jet bundle $J^k(E)$ of order $k$ is the bundle whose fiber over a point $x\in M$ consists of the equivalent classes of sections $s$ that agree up to order $k$ at $x$. A partial differential relation in this setting is a subset 
\[
\mathcal{R} \subset J^k(E),
\]
which imposes constraints on the $k$-jets of sections of $E$. A section $s:M \rightarrow E$ is a solution of the partial differential relation $\mathcal{R}$ if its $k$-jet $J^k_s$ satisfies
\[
J^k_s(x) \in \mathcal{R} \qquad \forall x \in M.
\]
This general framework encompasses a wide variety of partial differential relations arising in differential geometry. For further details, we refer the reader to Gromov's monograph \cite{gromov86}, as well as the books by Spring \cite{spring98} and Eliashberg and Mishachev \cite{em02}. Examples of first-order partial differential relations include various types of immersions between smooth manifolds, while second-order relations concern, for instance, free maps or Riemannian metrics satisfying specific curvature conditions. Several techniques have been developed to establish the existence of solutions to partial differential relations, sometimes at the cost of topological assumptions. Among these, the theory of convex integration -- developed by Gromov \cite{gromov73,gromov86}, following earlier work by Nash \cite{nash54} -- provides a method for resolving partial differential relations satisfying ampleness conditions. Nevertheless, some existence results are beyond the scope of currently available methods, such as those in Lohkamp's theory of negative Ricci or scalar curvature \cite{lohkamp94,lohkamp95,lohkamp99}. The aim of the present paper is to show how some of Lohkamp's results on scalar curvature can be recovered using a method of mixed convex integration.\\

Let $(M,g)$ be a smooth Riemannian manifold of dimension $n \geq 2$. In \cite{lohkamp95}, Lohkamp demonstrated that for $n\geq 3$, if $g_0$ is a smooth Riemannian metric on $M$, then for every $\alpha\in \R$ and every $\epsilon>0$, there exists a smooth Riemannian metric $g_{\epsilon}$ on $M$ such that
\[
r(g_{\epsilon}) < \alpha \quad \mbox{and} \quad \left\| g_{\epsilon} - g_0 \right\|_{C^0_g} < \epsilon,
\]
where $ r(g_{\epsilon})$ denotes the Ricci curvature of $g_{\epsilon}$ and $\|g_{\epsilon} - g_0\|_{C^0_g}$ is the $C^0$-norm of $g_{\epsilon}-g_0$ on $M$ with respect to the reference metric $g$. This result improved a previous result  by the same author in~\cite{lohkamp94}, which established the existence of smooth Riemannian metrics with negative Ricci curvature on any smooth manifold of dimension at least $3$. It is worth noting that, by the Gauss-Bonnet theorem, this result does not extend to dimension $2$. Lohkamp's proof is highly technical. It relies on a covering argument that superposes many metrics on $M$ derived from a perturbation of the Euclidean metric in $\R^n$, which is identical to the Euclidean metric outside the unit ball but exhibits negative Ricci curvature within the ball. Subsequently, Lohkamp refined his result in the case of scalar curvature. In \cite{lohkamp99}, he demonstrated the following result:

\begin{theorem}\label{THMScalar}
Let $(M,g)$ be a smooth Riemannian manifold of dimension $n\geq 3$, then for every smooth Riemannian metric $g_0$ on $M$, every smooth function $k: M \rightarrow (0,+\infty)$, and every $\epsilon >0$,  there exists a smooth Riemannian metric $g_{\epsilon}$ on $M$ satisfying 
\[
\mbox{\rm Scal}^{g_{0}}-k - \epsilon < \mbox{\rm Scal}^{g_{\epsilon}}<\mbox{\rm Scal}^{g_{0}}-k \quad \mbox{and} \quad \left\| g_{\epsilon} -g_0\right\|_{C^0_g} < \epsilon.
\]
\end{theorem}

It is important to note that the results of Lohkamp in \cite{lohkamp95,lohkamp99} are, in fact, stronger than those stated here. For instance,  in \cite{lohkamp99}, he showed that the new metric $g_{\epsilon}$ can be chosen to coincide with $g_0$ on a closed subset of $M$ while maintaining the same control on its scalar curvature as described in Theorem \ref{THMScalar} outside that set. The proof proceeds by superposing several small deformations of the initial metric, utilizing a suitable covering of the manifold. \\

The purpose of this paper is to develop a method of mixed convex integration and use it to provide a new proof of Theorem \ref{THMScalar}. The partial differential relation given by the scalar curvature does not fit within the general framework of convex integration mentioned above. Specifically,  it is not ample, as the scalar curvature tensor is affine in second-order partial derivatives. Our proof of Theorem \ref{THMScalar} exploits this property, along with the specific structure of first-order terms. It follows from a general method (see Proposition \ref{PROPmixedapp}) that allows solving a particular class of semilinear second-order partial differential relations. This mixed convex integration method relies on the corrugation process introduced by the third author in \cite{theilliere22}, which provides an alternative formula to those previously used in convex integration. It allows the perturbation to be split into several terms, each of which has a more easily controlled effect on the partial derivatives of the perturbation (see Proposition \ref{PROPmixed}). \\

The paper is organized as follows: Section \ref{SECmixed} introduces the general method of mixed convex integration and demonstrates its application for solving a particular class of second-order semilinear partial differential relations. Section \ref{SECAlmost} is dedicated to proving  Theorem \ref{THMScalar} in the case of a torus. For pedagogical clarity, we first address the perturbation of the flat metric on the torus, then explain how to handle a general metric on the torus. Section \ref{SEC4} provides the proof of Theorem \ref{THMScalar} in its general form. Finally, the Appendix includes the formulas of Ricci and scalar curvature in local coordinates, along with the proof of a technical result. \\

\medskip
\noindent
\textbf{Acknowledgment:} The third author was partially supported by the Luxembourg National Research Fund (FNR) O21/16309996/HypSTER, and then benefits from the support of the French government “Investissements d’Avenir” program integrated to France 2030, bearing the following reference ANR-11-LABX-0020-01.

\section{Mixed convex integration}\label{SECmixed}

The purpose of this section is to introduce a method of mixed convex integration taylored to solving a particular class of semilinear second-order partial differential relations. Instead of considering jet bundles between general manifolds, we focus on jets bundles from $\R^n$ to $\R^m$, allowing certain periodicity properties. This level of generality is sufficient for our application to Theorem \ref{THMScalar}. Throughout this section, we fix two integers $n$ and $m$ with $n,m\geq 1$.

\subsection{A few notations}

Given an integer $d$ such that $0\leq d\leq n$, and a set $K \subset \R^{n-d}$, we denote by $C^{\infty}_{per,d}(K \times \R^{d},\R^m)$ the set of smooth functions $f:K \times \R^{d} \rightarrow \R^m$ satisfying 
\[
f\left(x_1, \ldots, x_{n-d}, x_{n-d+1}+k_1, \ldots x_n+k_d\right) = f\left(x_1,  \ldots x_n\right) \qquad \forall x \in \R^n, \quad \forall k \in \Z^d,
\]
where $x=(x_1, \ldots, x_n)$ and $k=(k_1,\ldots, k_d)$. If $K$ is not open, a function $f:K \times \R^{d} \rightarrow \R^m$ is smooth if it extends smoothly to an open neighborhood of $K$. When $d=0$, the space $C^{\infty}_{per,d}(K \times \R^{d},\R^m)$ coincides with the set $C^{\infty}(K,\R^m)$ of smooth functions from $K \subset \R^n$ to $\R^m$. When $d=n$, it coincides with $C^{\infty}_{per}(\R^{n},\R^m)$, the set of smooth functions from $\R^n$ to $\R^m$ which are $\Z^n$-periodic. In all cases, $C^{\infty}_{per,d}(K \times \R^{d},\R^m)$ consists of smooth functions defined on a subset of $\R^n$ with coordinates $(x_1,\ldots,x_n)$ to $\R^m$ with coordinates $(y_1, \ldots,y_m)$. For any smooth function $f$ defined on a subset of $\R^n$ and valued in $\R^m$, we denote by $\partial_if$ and $\partial_{ij}^2f$ (with $i,j=1, \ldots,n$) the first and second-order partial derivatives of $f$ with respect to $x_i$ and $x_i,x_j$, respectively. If $i=j$ we sometimes use the notation $\partial^2_if$.

\subsection{Mixed corrugation processes}

Recall that $C^{\infty}_{per}(\R, \R^{m})$ denotes the set of smooth $1$-periodic functions from $\R$ to $\R^m$.

\begin{definition}[Integral-loop operator]
The integral-loop operator 
\[
\mbox{\rm Int} \, : \, C^{\infty}_{per}(\R, \R^{m}) \, \longrightarrow \, C^{\infty}_{per}(\R, \R^{m})
\]
is defined for every $\gamma \in C^{\infty}_{per}(\R, \R^{m})$ by
\[
\mbox{\rm Int} (\gamma)(t) := \int_0^t \left( \gamma(s)-\overline{\gamma}\right) \, ds \quad \forall t \in \R \quad \mbox{with} \quad \overline{\gamma} := \int_0^1 \gamma(s) \, ds,
\]
and for every positive integer $n$ we denote by $\mbox{\rm Int}^n= \mbox{\rm Int} \circ \cdots \circ \mbox{\rm Int}$ the mapping corresponding to $n$ iterations of $\mbox{\rm Int}$.
\end{definition}

The integral-loop operator defines a linear operator from the vector space $C^{\infty}_{per}(\R, \R^{m})$ into itself, it is the building block of corrugation processes, the deformation technique used in convex integration. \\

Fix an integer $d$ with $0\leq d\leq n$ and a compact set $K \subset \R^{n-d}$, and define $E:=K\times \R^d$. Given $f\in C^{\infty}_{per,d}(E,\R^m)$ and $\{\gamma_x\}_{x\in E}$ a well-adapted family of loops valued in $\R^m$, that is, of the form $\{\gamma_x:=\gamma(x,\cdot)\}_{x\in E}$ for some $\gamma \in C^{\infty}_{per,d+1}(K\times \R^{d+1},\R^m)$, we define, for any positive integers $N, b$, and any $i\in \{1, \ldots, n\}$, the smooth function $F\in C^{\infty}_{per,d}(E,\R^m)$ by 
\[
F(x) := f(x) + \frac{1}{N^b} \, \mbox{Int}^b \left( \gamma_x\right) \left( Nx_i \right) \qquad \forall x =\left( x_1, \ldots, x_n \right) \in K.
\]
The function $F$, said to be obtained by a corrugation process from $f$ for $\partial_i^b$, tends to $f$ in $C^0$-topology as $N \rightarrow +\infty$, its derivatives of order strictly less than $b$ tend to the corresponding derivatives of $f$, and the derivative of order $b$ along the $x_i$ variable are asymptotically controled by that of $f$ and $\gamma$ (see Proposition \ref{PROPmixed} (i)-(iii) below with $\delta_x\equiv 0$). We now combine corrugation processes at orders one and two. The following result describes the properties of smooth functions obtained through this mixed corrugation process. 

\begin{proposition}\label{PROPmixed}
Let $f\in C^{\infty}_{per,d}(E,\R^m)$ and $ \{\gamma_x\}_{x\in E}$, $\{\delta_x\}_{x\in E}$ be two well-adapted smooth families of loops valued in $\R^m$. For any positive integer $N$ and any $i\in \{1, \ldots, n\}$, define the function $F\in C^{\infty}_{per,d}(E,\R^m)$ by
\begin{eqnarray}\label{PROPmixedEQ}
\qquad F(x) := f(x) + \frac{1}{N} \, \mbox{\rm Int} \left( \gamma_x\right) \left( Nx_i \right) + \frac{1}{N^2} \, \mbox{\rm Int}^2 \left( \delta_x\right) \left( Nx_i \right) \qquad \forall x =\left( x_1, \ldots, x_n \right) \in E.
\end{eqnarray}
Then, the following properties are satisfied:
\begin{itemize}
\item[(i)] $\|F-f\|_{C^0} = O(1/N)$ (as $N \rightarrow + \infty$).
\item[(ii)] For all $j \in \{1, \ldots,n \}$ with $j\neq i$, $\|\partial_{j}F-\partial_{j}f\|_{C^0} = O(1/N)$.
\item[(iii)] For all $j,k \in \{1, \ldots,n \}$ with $j,k\neq i$, $\|\partial_{jk}^2F-\partial_{jk}^2f\|_{C^0} = O(1/N)$.
\end{itemize}
Moreover, the following properties are satisfied for any $x\in E$, with $O(1/N)$ that is uniform with respect to $x\in E$:
\begin{itemize}
\item[(iv)] $\partial_iF(x)=\partial_if(x)+ \gamma_x(Nx_i) - \overline{\gamma_x} +O(1/N)$.
\item[(v)] $\partial_i^2 F(x)=\partial_i^2 f(x)+ 2 \left( \partial_i \gamma_x(Nx_i) - \overline{(\partial_i \gamma_x)}\right) + N \dot{\gamma}_x(Nx_i) + \delta_x(Nx_i) - \overline{\delta_x} + O(1/N)$, \\
where $\dot{\gamma}_x$ stands for  the derivative of $\gamma_x$ with respect to the periodic parameter of the loop.
\item[(vi)] For all $j \in \{1, \ldots,n \}$ with $j\neq i$, $\partial_{ij}^2 F(x)=\partial_{ij}^2 f(x)+  \partial_j \gamma_x(Nx_i) - \overline{(\partial_j \gamma_x)}  + O(1/N)$.
\end{itemize}
\end{proposition}

\begin{proof}[Proof of Proposition \ref{PROPmixed}]
To prove (i), we observe that for every $x\in E$, we have 
\[
F(x) - f(x) = \frac{1}{N} \, \mbox{\rm Int} \left( \gamma_x\right) \left( Nx_i \right) + \frac{1}{N^2} \, \mbox{\rm Int}^2 \left( \delta_x\right) \left( Nx_i \right).
\]
By compactness and periodicity, the functions $(x,t) \in K \times \R^{d+1}\mapsto \gamma_x(t)$ and  $(x,t) \in K \times \R^{d+1} \mapsto \delta_x(t)$ are bounded. Consequently, the functions $(x,t) \in K \times \R^{d+1}\mapsto \mbox{Int}(\gamma_x)(t)$ and $(x,t) \in K \times \R^{d+1}\mapsto \mbox{Int}(\delta_x)(t)$ are also bounded. Thus, we deduce that $\|F-f\|_{C^0} = O(1/N)$.

To prove properties (ii)-(v), we note that, if $\{\alpha_x\}_{x\in K}$ is a smooth family of loops then 
\begin{eqnarray}\label{derivative_loop}
\partial_j \left( \mbox{Int} \left(\alpha_x\right) (t)\right) =  \mbox{Int} \left(\partial_j \alpha_x\right) (t) 
\qquad \mbox{and} \qquad 
\partial_{jk} \left( \mbox{Int} \left(\alpha_x\right) (t)\right) =  \mbox{Int} \left(\partial_{jk} \alpha_x\right) (t) 
\end{eqnarray}
$\forall x \in E, \, \forall t \in \R, \, \forall j,k\in \{1,\ldots,n\}$.

To prove (ii), we fix $j\in \{1,\ldots,n\}$ with $j\neq i$ , and evaluting $\mbox{Int}( \gamma_x)$ and $\mbox{Int}^2( \delta_x)$ in \eqref{derivative_loop} at $t=Nx_i$, whose derivative in $\partial_j$ is zero, we obtain 
\[
\partial_j F(x) - \partial_j f(x) = \frac{1}{N} \, \mbox{\rm Int} \left(\partial_j\gamma_x\right) \left( Nx_i \right) + \frac{1}{N^2} \, \mathrm{Int}^2 \left(\partial_j \delta_x\right) \left( Nx_i \right) \qquad \forall x \in E.
\]
Since the functions $(x,t) \mapsto \mbox{Int}( \partial_j \gamma_x)(t)$ and $(x,t) \mapsto \mbox{Int}^2 \left( \partial_j \delta_x\right)(t)$ are bounded over $K\times \R^{d+1}$, by the same argument as in (i), we infer that $\|\partial_{j}F-\partial_{j}f\|_{C^0} = O(1/N)$.

To prove (iii), we fix $j,k \in \{1, \ldots,n \}$ with $j,k\neq i$. In the same way as in (ii), as $j,k\neq i$, we have for every $x\in E$
\[
\partial_{jk}^2 F(x) - \partial_{jk}^2 f(x) = \frac{1}{N} \, \mbox{\rm Int} \left(\partial_{jk}^2 \gamma_x\right) \left( Nx_i \right) + \frac{1}{N^2} \, \mathrm{Int}^2 \left(\partial_{jk}^2 \delta_x\right) \left( Nx_i \right) \qquad \forall x \in E.
\]
We conclude as before. 

To prove (iv) and (v), we use again the formula \eqref{derivative_loop}, but this time, as we consider $t=Nx_i$, and take derivatives in $\partial_i$, we have an additional term. We obtain for every $x\in E$, 
\begin{multline*}
\partial_i F(x)
= \partial_i f(x) + \gamma_x(Nx_i)-\overline{\gamma_x} \\
+\frac{1}{N} \, \mbox{Int} \left( \partial_i\gamma_x\right) \left( Nx_i \right) + \frac{1}{N} \left( \mbox{Int} \left(\delta_x\right) \left( Nx_i \right) - \overline{\mbox{Int} (\delta_x)} \right) + \frac{1}{N^2} \, \mbox{Int}^2 \left(\partial_i \delta_x\right) \left( Nx_i \right)
\end{multline*}
and
\begin{multline*}
\partial_i^2 F(x) = \partial_i^2 f(x) + N\dot{\gamma}_x \left(Nx_i\right) +2\left( \left(\partial_i\gamma_x\right)\left(Nx_i\right)-\overline{ \partial_i\gamma_x} \, \right) +\frac{1}{N} \, \mbox{Int} \left( \partial_i^2\gamma_x\right) \left( Nx_i \right)\\
 + \delta_x \left(Nx_i\right) - \overline{\delta_x} + \frac{2}{N}\left( \mbox{Int} \left(\partial_i\delta_x\right) \left(Nx_i\right) - \overline{\mbox{Int} \left(\partial_i\delta_x\right)} \right) + \frac{1}{N^2} \, \mbox{Int}^2 \left(\partial_i^2 \delta_x\right) \left( Nx_i \right).
\end{multline*}
We conclude as above. Finally, assertion (vi) follows by taking the partial derivative with respect to $x_j$ of the above formula for $\partial_i F$.
\end{proof}

\subsection{Semilinear second-order partial differential relations}\label{SECsemilinear}

As before, we fix an integer $d$ with $0\leq d\leq n$ and a compact set $K \subset \R^{n-d}$, and define $E:=K\times \R^d$.  A partial differential relation $\mathcal{R}$ in $J^2(E,\R^m)$ is a subset of
\[
J^2(E,\R^m) = E \times \R^m \times \left(\R^{m}\right)^n \times \left(\R^{m}\right)^{\frac{n(n+1)}{2}}
\]
 and a smooth function $f\in C^{\infty}_{per,d} (E,\R^m)$ is a solution of $\mathcal{R}$ if its $2$-jet satisfies
\[
J^2_f(x) :=  \left(x,f(x), \left(\partial_1 f(x), \cdots, \partial_nf(x)\right), \left(\partial_{ij}^2f(x)\right)_{1\leq i\leq j\leq n}\right) \in \mathcal{R} \qquad \forall x \in E.
\]
Our goal is to show how the mixed corrugation process introduced in the previous section can be used to solve a particular class of semilinear second-order partial differential relations. For simplicity, we focus on partial differential relations that are solvable by a mixed corrugation process along the $x_1$ variable.\\

To isolate the components of the $2$-jets corresponding to derivatives along the $x_1$ variable, we write the elements of $J^2(E,\R^m)$ as
\[
\sigma = \left(\sigma_0,\sigma^0, \sigma^1_1, \cdots, \sigma^1_n, \left(\sigma^2_{ij}\right)_{1\leq i\leq j\leq n}\right). 
\]
We denote by $\check{\sigma}$ the tuple obtained from $\sigma$ by removing the terms $\sigma_0$, $\sigma^1_1$ and $\sigma^2_{1j}$ for $j=1,\ldots,n$, and define $\check{J}^2(E,\R^m)$ as the set of all such tuples $\check{\sigma}$ with $\sigma \in J^2(E,\R^m)$. Finally, we use the notation $\sigma^2_{1*}$ to refer to the tuple of all components of the form $\sigma^2_{1j}$ with $j=2, \ldots,n$.

\begin{definition}\label{DEFRsemilinear}
A partial differential relation $\mathcal{R} \subset J^2(E,\R^m)$ is said to be semilinear in $x_1$ if there are two smooth maps $L:E \rightarrow (\R^m)^*$ and  $R : E \times \check{J}^2(E,\R^m) \times \R^m \times (\R^m)^{n-1} \rightarrow \R$ satisfying 
\[
\left\{
\begin{array}{l}
L\left(\sigma_0+ (0_{n-d},k)\right) = L\left(\sigma_0\right)\\
R\left(\sigma_0+ (0_{n-d},k), \check{\sigma}, \sigma_1^1, \sigma^2_{1*} \right) = R\left(\sigma_0, \check{\sigma}, \sigma_1^1, \sigma^2_{1*} \right)
\end{array}
\right.
\qquad \forall \left(0_{n-d},k\right)\in \{0\}\times \Z^d
\]
for all $\sigma_0 \in E$, $\check{\sigma}\in \check{J}^2(E,\R^m)$, $\sigma_1^1 \in \R^m$, and $\sigma^2_{1*} \in  (\R^m)^{n-1}$, such that
\[
\mathcal{R} := \Bigl\{  \sigma \in J^2(E,\R^m) \, \vert \,  L\left(\sigma_0\right) \cdot \sigma^2_{11} +   R\left(\sigma_0, \check{\sigma}, \sigma_1^1, \sigma^2_{1*} \right)=0 \Bigr\}.
\]
For every $\epsilon>0$, we call $\epsilon$-thickening of $\mathcal{R}$ the open partial differential relation  $\mathcal{R}_{\epsilon} \subset J^2(E, \R^m)$ defined by 
\[
\mathcal{R}_{\epsilon} := \Bigl\{  \sigma \in J^2(E, \R^m) \, \vert \, L\left(\sigma_0\right) \cdot \sigma^2_{11} +  R\left(\sigma_0, \check{\sigma}, \sigma_1^1 , \sigma^2_{1*}\right) \in (-\epsilon, \epsilon) \Bigr\}.
\]
\end{definition}

The reason for choosing to include $\check{\sigma}, \sigma_1^1 , \sigma^2_{1*}$ in the definition of $R$ is justified by the following result, whose proof follows from Proposition \ref{PROPmixed}.

\begin{proposition}\label{PROPmixedapp}
Let $\mathcal{R} \subset J^2(E, \R^m)$ be a semilinear (in $x_1$) partial differential relation as in Definition \ref{DEFRsemilinear}, and let $f \in C^{\infty}_{per,d}(E,\R^m)$. Assume that there are  two well-adapted smooth families of loops $\{\gamma_x\}_{x\in E}$ and $\{\delta_x\}_{x\in E}$ valued in $\R^m$ satisfying the following properties:
\begin{itemize}
\item[(L1)] For every $x\in E$, $\overline{\gamma}_x=\overline{\delta}_x=0$;
\item[(L2)] For every $x\in E$ and every $t\in \R$, $L(x) \cdot \gamma_x(t) = \partial_1L(x) \cdot \gamma_x(t)= 0$. 
\item[(L3)]  For every $x\in E$ and every $t\in \R$, 
\[
L (x) \cdot \left(\sigma_{11}^2 + \delta_x(t)\right) +  R\left(x,\check{\sigma}, \sigma_1^1(x) +\gamma_x(t), \sigma^2_{1*} +\partial_{*} \gamma_x(t)  \right) = 0,
\]
where $\sigma:= J^2_f(x)$ and $\partial_{*} \gamma_x(t)$ is the tuple $(\partial_2\gamma_x(t), \cdots, \partial_n\gamma_x(t))$.
\end{itemize}
then, for every $\epsilon >0$, the smooth function $F\in C^{\infty}_{per,d} (E, \R^m)$ given by formula (\ref{PROPmixedEQ}) satisfies the following properties for $N$ sufficiently large:
\begin{itemize}
\item[(a)] $\|F-f\|_{C^0} <\epsilon$;
\item[(b)] For every $x\in E$, $J^2_F(x) \in \mathcal{R}_{\epsilon}$.
\end{itemize}
\end{proposition}

\begin{proof}[Proof of Proposition \ref{PROPmixedapp}]
Let $\mathcal{R}$ and $f$ be as in the statement, and let $\{\gamma_x\}_{x\in E}$ and $\{\delta_x\}_{x\in E}$  be two well-adapted smooth families of loops valued in $\R^m$ satisfying (L1)-(L3). Assertion (a) follows from Proposition \ref{PROPmixed} (i). To prove (b), fix $x\in E$ and set 
\[
\sigma:= J_f^2(x) \quad \mbox{and} \quad \theta:= J_F^2(x). 
\]
Assertions (ii)-(iii) of Proposition \ref{PROPmixed} yield
\[
\check{\theta} = \check{\sigma} + O(1/N),
\]
where $O(1/N)$, here and throughout, is uniform with respect to $x \in E$. Moreover, assertions (iv)-(vi) of Proposition \ref{PROPmixed}, combined  with (L1), give
\begin{eqnarray*}
\left\{
\begin{array}{rcl}
\theta_1^1 = \partial_1F(x) & = & \sigma_1^1+ \gamma_x\left(Nx_1\right)  +O(1/N)\\
\theta_{11}^2 =\partial_1^2 F(x) & =& \sigma_{11}^2 + 2 \partial_1 \gamma_x\left(Nx_1\right) + N \dot{\gamma}_x \left(Nx_1\right) + \delta_x\left(Nx_1\right) + O(1/N)\\
\theta_{1j}^2  = \partial_{ij}^2 F(x) & = & \sigma_{1j}^2+  \partial_j \gamma_x(Nx_1)  + O(1/N), \, \forall j=2, \ldots,n.
\end{array}
\right.
\end{eqnarray*}
The derivation of the first equality of (L2) with respect to $t$ and $x_1$, along with its second term, gives
\[
L(x) \cdot  \dot{\gamma}_x (t) = L(x) \cdot  \partial_1 \gamma_x(t)  = 0 \qquad \forall t \in \R.
\]
As a result, by setting $\theta^2_{1*}:=(\theta^2_{12}, \cdots, \theta^2_{1n})$, we have
\begin{multline*}
L (x) \cdot  \theta_{11}^2 + R \left(x, \check{\theta}, \theta_1^1, \theta^2_{1*} \right) = L(x) \cdot  \left( \sigma_{11}^2 + \delta_x\left(Nx_1\right) + O(1/N) \right) \\
+ R \left(x, \check{\sigma} + O(1/N), \sigma_1^1+ \gamma_x\left(Nx_1\right)  +O(1/N), \sigma^2_{1*}+ \partial_{*} \gamma_x\left(Nx_1\right)  + O(1/N)\right).
\end{multline*} 
Next, observe that, by the periodicity properties of $L$ and the compactness of $K$, 
\[
L(x) \cdot \left( \sigma_{11}^2 + \delta_x\left(Nx_1\right) + O(1/N) \right) = L(x) \cdot \left( \sigma_{11}^2 + \delta_x\left(Nx_1\right)  \right) + O(1/N),
\]
and, by the periodicity properties of $f$, $\{\gamma_x\}_{x\in E}$, $\{\delta_x\}_{x\in E}$ and $R$, and the compactness of $K$, that
\begin{multline*}
 R \left(x, \check{\sigma} + O(1/N), \sigma_1^1+ \gamma_x\left(Nx_1\right)  +O(1/N),  \sigma^2_{1*}+ \partial_{*} \gamma_x\left(Nx_1\right) +O(1/N)\right) \\
 =  R \left(x, \check{\sigma}, \sigma_1^1+ \gamma_x\left(Nx_1\right),  \sigma^2_{1*}+ \partial_{*} \gamma_x\left(Nx_1\right) \right) + O(1/N).
\end{multline*}
By (L3), it follows that 
\[
L (x) \cdot \theta_{11}^2 + R \left(x, \check{\theta}, \theta_1^1, \theta^2_{1*} \right)=O(1/N),
\]
which completes the proof of (b).
\end{proof}

\subsection{Geometrical interpretation of Proposition~\ref{PROPmixedapp}}\label{subsection:GeomInterpretation}

The constructions of solutions in classical convex integration theory relies on a single family of loops that satisfy a specific property, which can be interpreted geometrically. Here, we provide the geometric interpretation related to the existence of the loops families $\{\gamma_x\}_{x\in E}$ and $\{\delta_x\}_{x \in E}$ from Proposition~\ref{PROPmixedapp}.\\

Let us consider the framework of Proposition~\ref{PROPmixedapp} and assume, in addition, that $\mathcal{R}$ does not depend upon $\sigma_{1*}^2$ and that $L$ is constant. The condition (L2) is equivalent to the requirement that, for any $x\in E$, the image of $\gamma_x$ lies within the vector subspace $\mbox{Ker}(L)\subset \R^m$. The condition (L3) involves finding two well-adapted smooth families of loops, $\{\gamma_x\}_{x\in E}$ and $\{\delta_x\}_{x\in E}$, such that for each $x\in E$, the map  $t\mapsto (\gamma_x(t), \delta_x(t))$ takes values in the set
\begin{eqnarray*}
\mathcal{R}_{\sigma, \partial_1, \partial^2_1} := \Bigl\{(v,w)\in \R^m \times \R^m \;|\; L \cdot \left(\sigma_{11}^2 + w\right)  +   R\left(x,\check{\sigma}, \sigma_1^1(x) + v \right) = 0 \Bigr\}.
\end{eqnarray*}
Furthermore, the condition (L1) imposes a constraint on the averages of the maps $t\in \R\mapsto (\gamma_x(t), \delta_x(t))$ for $x\in E$. In particular, the assumptions of Proposition~\ref{PROPmixedapp} cannot be satisfied if the convex hull of the set $(\mbox{Ker}(L) \times \R^m ) \cap \mathcal{R}_{\sigma, \partial_1, \partial^2_1} \subset \R^m \times \R^m$ does not contain the origin. This geometric interpretation will be discussed through a specific example in Remark~\ref{REMInterpretation}.

\section{Almost prescribing scalar curvature on tori}\label{SECAlmost}

The purpose of this section is to explain how techniques of mixed convex integration can be employed to construct smooth metrics with almost prescribed curvature on tori of dimension at least $3$. For clarity, we begin with the case of a perturbation of the flat metric on $\T^n$. We then address the case of a general metric on $\T^n$. Finally, we describe how to modified the proof of the latter to construct a metric with almost prescribed scalar curvature on a thick torus. This result will play a key role in the proof of Theorem \ref{THMScalar}.

\subsection{Perturbations of the flat metric}\label{subsection:perturb_flat_metric}

We prove Theorem \ref{THMScalar} here in the case where $g_0=g$ is the flat metric induced by the Euclidean metric on $\T^n=\R^n/\Z^n$. In this setting, we have $\mbox{Scal}(g_0)=0$.

\begin{proposition}\label{PROPTnFlat}
Assume that $n\geq 3$. Then, for every smooth function $k:\T^n \rightarrow (0,+\infty)$ and every $\epsilon >0$, there exists a smooth Riemannian metric $g_{\epsilon}$ on $\T^n$ satisfying 
\[
-k-\epsilon < \mbox{\rm Scal}^{g_{\epsilon}} < -k \quad \mbox{and} \quad \left\| g_{\epsilon}-g_0 \right\|_{C^0_g} < \epsilon.
\]
\end{proposition}

\begin{proof}[Proof of Proposition \ref{PROPTnFlat}]
Let $k:\T^n \rightarrow (0,+\infty)$ be a smooth function, and let $\epsilon >0$ be fixed. Set $\bar{\epsilon}:= \epsilon/2$. Denote by $\tilde{g}_0$ and $\tilde{k}$ the lifts of $g_0$ and $k$ to $\R^n$ as $\Z^n$-periodic metrics and functions, respectively. We consider a perturbation $\tilde{g}$ of $\tilde{g}_0$ whose matrix $(\tilde{g}_{ij})$ has the form ($\delta_{ij}$ denotes the Kronecker symbol)
\[
\left\{
\begin{array}{l}
\tilde{g}_{ij}  = \left(\tilde{g}_0\right)_{ij} = \delta_{ij} \mbox{ for } (i,j) \in \{1,\ldots,n\}^2 \setminus \{(2,2),(3,3)\}\\
\tilde{g}_{ii}  = e^{2h_i} \left(\tilde{g}_0\right)_{ii}= e^{2h_i} \mbox{ for } (i,j) \in \{(2,2),(3,3)\},
\end{array}
\right.
\]
where $h_2, h_3: \R^n \rightarrow \R$ are two smooth $\Z^n$-periodic functions. The scalar curvature of $\tilde{g}$ is given by the expression (see Appendix \ref{APPgdiag})
\begin{multline*}
\mbox{Scal}^{\tilde{g}} =  - 2 \Bigl(   \partial_{1}^2 h_2  +   \partial_{1}^2 h_3  + \left( \partial_{1}h_2 \right)^2 +  \left( \partial_{1}h_3 \right)^2 + \partial_{1} h_2 \, \partial_1 h_3\Bigr)\\
 -2e^{-2h_2} \Bigl( \partial_2^2 h_3 -\partial_2h_2 \, \partial_2h_3 +(\partial_2h_3)^2\Bigr) -2e^{-2h_3} \Bigl( \partial_3^2 h_2 -\partial_3h_3 \, \partial_3h_2 +(\partial_3h_2)^2\Bigr)\\
 - 2 \sum_{i=4}^n \Bigl( \partial_i^2h_2+\partial_i^2h_3 + (\partial_ih_2)^2 + (\partial_ih_3)^2 - \partial_ih_2 \, \partial_ih_3 \Bigr).
\end{multline*}
Consequently, the desired result follows, by setting $h_2=F_1$ and $h_3=F_2$, if we show that there is a function $F=(F_1,F_2)\in C^{\infty}_{per}(\R^n,\R^2)$ such that 
\[
\left\| F\right\|_{C^0} < \bar{\epsilon} \quad \mbox{and} \quad  J^2_F(x) \in \mathcal{R}_{\bar{\epsilon}},
\]
where $\mathcal{R}\subset J^2(\R^n,\R^2)$ is the semilinear (in $x_1$) partial differential relation  defined as
\[
\mathcal{R} := \Bigl\{  \sigma \in J^2(\R^n,\R^2) \, \vert \,  L\cdot \sigma^2_{11} +   R\left(\sigma_0, \check{\sigma},\sigma_1^1,\sigma^2_{1*}\right)=0 \Bigr\}.
\]
with 
\begin{eqnarray}\label{18Avril_1}
L = -2(1,1)
\end{eqnarray}
and
\begin{multline}\label{18Avril_2}
R\left(\sigma_0, \check{\sigma}, \sigma_1^1, \sigma^2_{1*}\right) = -2 \Bigl(\left(\sigma^1_1\right)_1^2 + \left(\sigma^1_1\right)_2^2 + \left(\sigma^1_1\right)_1 \left(\sigma_1^1\right)_2 \Bigr)+ \tilde{k} \left(\sigma_0\right) + \bar{\epsilon}\\
-2e^{-2(\sigma^0)_1} \Bigl( \left(\sigma_{22}^2\right)_2 -  \left(\sigma^1_2\right)_1 \,  \left(\sigma^1_2\right)_2 +  \left(\sigma^1_2\right)_2^2\Bigr) -2e^{-2(\sigma^0)_2} \Bigl( \left(\sigma_{33}^2\right)_1 -  \left(\sigma^1_3\right)_2 \,  \left(\sigma^1_3\right)_1 +  \left(\sigma^1_3\right)_1^2\Bigr) \\
- 2 \sum_{i=4}^n \Bigl( \left(\sigma_{ii}^2\right)_1 + \left(\sigma_{ii}^2\right)_2  +  \left(\sigma^1_i\right)_1^2 +  \left(\sigma^1_i\right)_2^2 - \left(\sigma^1_i\right)_1 \,  \left(\sigma^1_i\right)_2\Bigr).
\end{multline}
To apply Proposition \ref{PROPmixedapp} with $f\equiv 0$, we need to construct two well-adapted ($\Z^n$-periodic in the $x$ variable) smooth families of loops $\{\gamma_x\}_{x\in \R^n}$ and $\{\delta_x\}_{x\in \R^n}$ valued in $\R^2$ such that, for every $x\in \R^n$, the following conditions are satisfied:
\begin{multline*}
\overline{\gamma}_x=\overline{\delta}_x=0, \quad (\gamma_x)_1 + (\gamma_x)_2 = 0, \\
\quad \mbox{and} \quad (\delta_x)_1 + (\delta_x)_2  = \frac{\tilde{k}(x)+\bar{\epsilon}}{2} - ((\gamma_x)_1)^2 - ((\gamma_x)_2)^2 -  ((\gamma_x)_1) ((\gamma_x)_2). 
\end{multline*}
From these conditions, we must have $(\gamma_x)_2 = -(\gamma_x)_1$. If we further impose $(\delta_x)_2=0$, then $(\delta_x)_1$ must be defined as 
\begin{eqnarray}\label{EQ24janv1}
 (\delta_x)_1 := \frac{\tilde{k}(x)+\bar{\epsilon}}{2} - ((\gamma_x)_1)^2,
 \end{eqnarray}
and it must satisfy $\overline{\left(\delta_x\right)_1 } =0$. Since $\int_0^1 \cos^2 (2\pi t) \, dt=1/2$, the function  
\[
(\gamma_x)_1(t) := \sqrt{\tilde{k}(x)+\bar{\epsilon}} \, \cos (2\pi t) \qquad \forall x  \in \T^n, \, \forall t \in \R
\]
ensures that $(\delta_x)_1$, as defined in (\ref{EQ24janv1}), satisfies $\overline{\left(\delta_x\right)_1 } =0$. We conclude the proof by considering the metric on $\T^n$ obtained from $\tilde{g}$ via the quotient operation and by applying Proposition~\ref{PROPmixed} with $N$ large enough.
\end{proof}

\begin{remark}\label{REMInterpretation}
Let us pursue the geometric interpretation initiated in Section \ref{subsection:GeomInterpretation}. The partial differential relation that appears in the above proof is given by 
\[
\mathcal{R} = \Bigl\{  \sigma \in J^2(\R^n,\R^2) \, \vert \,  L\cdot \sigma^2_{11} +   R\left(\sigma_0, \check{\sigma},\sigma_1^1,\sigma^2_{1*}\right)=0 \Bigr\}.
\]
where $L$ and $R$ are defined in equations (\ref{18Avril_1}) and (\ref{18Avril_2}). The existence of well-adapted ($\Z^n$-periodic in the $x$ variable) families of loops $\{\gamma_x\}_{x\in \R^n}$ and $\{\delta_x\}_{x\in \R^n}$, satisfying conditions (L1)-(L3) of Proposition \ref{PROPmixedapp} with $\sigma=0$ requires the existence, for each $x\in \R^n$, of a loop $(\gamma_1,\gamma_2,\delta_1,\delta_2)\in C^{\infty}_{per}(\R, \R^{4})$ taking values in the set $\mbox{\rm Ker}(L) \cap \mathcal{R}_{\sigma,\partial_1,\partial^2_1}$, where
\begin{eqnarray*}
\mbox{\rm Ker}(L) & = & \Bigl\{\left(v_1,v_2,w_1,w_2\right)\in \R^4 \, | \, v_1+v_2 =0\Bigr\}\\
\mathcal{R}_{\sigma,\partial_1,\partial^2_1} &=& \Bigl\{\left(v_1,v_2,w_1,w_2\right)\in \R^4 \, | \, -2[ w_1 + w_2 + v_1^2 + v_2^2 + v_1v_2 ] + \tilde{k}(x) + \bar{\epsilon} =0 \Bigr\},
\end{eqnarray*}
and whose average equals zero (condition $(L_1)$). To visualize it in the hyperplane $Ker(L)$,  set $v_1=-v_2$; then $\mathcal{R}_{\sigma,\partial_1,\partial^2_1}$ corresponds to the surface below: 

\begin{figure}[H]
\begin{center}
\includegraphics[width=6cm]{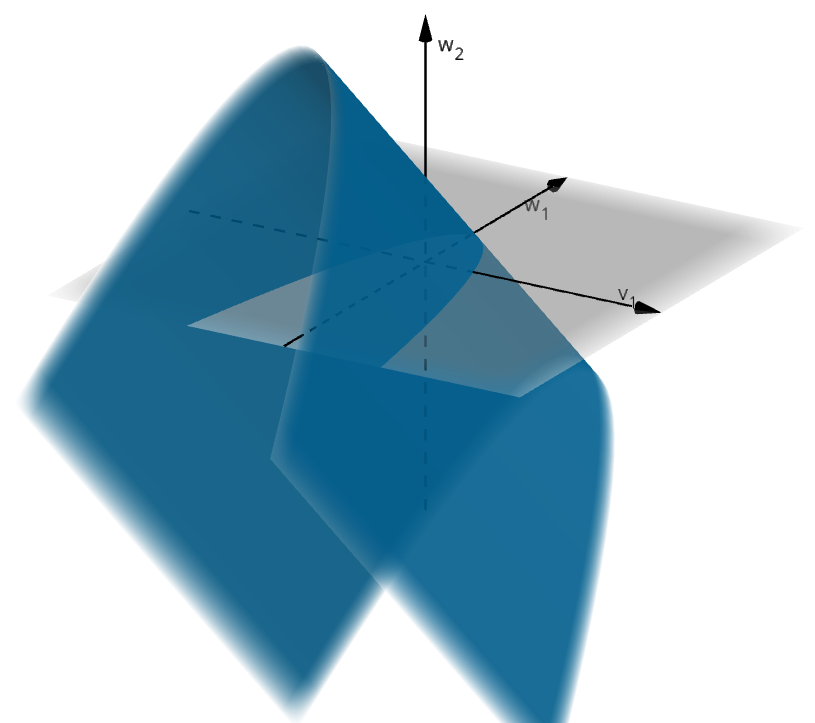}
\caption{The surface $2(w_1+w_2+v_1^2)=F $, where $F$ is a constant\label{fig1}}
\end{center}
\end{figure}

Since the function $H(v_1,w_1,w_2)=2(w_1+w_2+v_1^2)$ is convex, the point $(0,0,0)$ belongs to the convex hull of a level set $\{H=F\}$ for some $F\in \R$ if and only if $0=H(0,0,0)\leq F$. This implies that condition (L1) requires $\tilde{k}(x)+\epsilon$ to be nonnegative, which essentially explains why our perturbation method does not accomodate negative values of  $k$.  
\end{remark}

\subsection{Perturbations of a general metric on $\T^n$}

We prove Theorem \ref{THMScalar} here in the case of a general metric on $\T^n$. The main idea of the proof, aside from the use of Proposition \ref{PROPmixedapp}, is to consider a certain type of diagonal deformation of the metric, up to a diffeomorphism. The metric $g$ denotes the flat metric induced by the Euclidean metric on $\T^n=\R^n/\Z^n$.

\begin{proposition}\label{PROPTnGen}
Assume that $n\geq 3$. Then, for every smooth Riemannian metric $g_0$ on $\T^n$, every smooth function $k:\T^n \rightarrow (0,+\infty)$, and every $\epsilon >0$, there exists a smooth Riemannian metric $g_{\epsilon}$ satisfying 
\[
-k-\epsilon < \mbox{\rm Scal}^{g_{\epsilon}} - \mbox{\rm Scal}^{g_0}< -k \quad \mbox{and} \quad \left\| g_{\epsilon} -g_0\right\|_{C^0_g} < \epsilon.
\]
\end{proposition}

\begin{proof}[Proof of Proposition \ref{PROPTnGen}]
Let $g_0$ be a smooth Riemannian metric on $\T^n$, $k:\T^n \rightarrow (0,+\infty)$ be a smooth function, and let $\epsilon >0$ be fixed. 
Denote by $\tilde{g}_0$ and $\tilde{k}$ the lifts of $g_0$ and $k$ to $\R^n$ as $\Z^n$-periodic metrics and functions, respectively. We consider a $\Z^n$-periodic smooth orthonormal family of vector fields $X^1, \ldots, X^n$ on $\R^n$, that is, satisfying ($\delta_{ij}$ stands for the Kronecker symbol)
\[
\left(\tilde{g}_0\right)_x \left( X^{i},X^j\right) = \delta_{ij} \qquad \forall x \in \R^n,
\]
and we assume that 
\begin{eqnarray*}
\alpha_x:=\langle X^1(x),e_1\rangle  >0 \quad \mbox{and} \quad \langle X^2(x),e_1\rangle= \cdots = \langle X^n(x),e_1\rangle=0 \qquad \forall x \in \R^n,
\end{eqnarray*}
where $e_1, \ldots, e_n$ and $\langle \cdot, \cdot \rangle$ stand respectively for the canonical basis and the canonical scalar product in $\R^n$. Such a family can easily be constructed by applying the Gram-Schmidt process starting from the canonical orthonormal family for the Euclidean metric. Then, for every $x\in \R^n$ we denote by $P_x=((P_x)_{ij})$ the $n\times n$ matrix with columns $X^1(x), \ldots, X^n(x)$ which satisfies
\begin{eqnarray}\label{5fev1}
P_x^\mathsf{tr}\left(\tilde{g}_0\right)_xP_x=\mbox{Id}_n 
\end{eqnarray}
and
\begin{eqnarray}\label{5fev2}
 \left(P_x\right)_{1,j} = \delta_{1j} \alpha_x \quad \forall j=1, \ldots,n,
\end{eqnarray}
and we define a perturbed metric $\tilde{g}$ on $\R^n$ by 
\begin{eqnarray}\label{gtildeformula}
\tilde{g}_x = Q_x^{\mathsf{tr}}D(x)Q_x \quad \mbox{with} \quad Q_x:=P_x^{-1}   \qquad \forall x \in \R^n,
\end{eqnarray}
where $D$ is a $\Z^n$-periodic $n\times n$ diagonal matrix with coefficients $(1, e^{2h_2}, e^{2h_3},1, \ldots,1)$ with $h_2, h_3 : \R^n \rightarrow \R$ two smooth functions to be chosen later. The proof of the following lemma is postponed to Appendix \ref{PROPTnGenLEMProof}. It relies on the formulas for scalar curvature provided in Proposition \ref{PROPgorthonormal}, which appear in Appendix \ref{APPPertMetric}, and (\ref{5fev1})-(\ref{gtildeformula}). In the statement, $\check{J}_{h}^2(x)$ represents the $2$-jet of the function $h=(h_1,h_2)$ with the components $x$, $\partial^1_1h$ and $\partial^2_{1j}h$ for $j=1,\ldots,n$ removed, as was done in Section \ref{SECsemilinear}. Additionally, we denote by $0$ the elements of $\check{J}^2(\T^n, \R^2)$ for which all components are zero. 

\begin{lemma}\label{PROPTnGenLEM}
For every $x\in \R^n$, we have 
\begin{multline*}
\mbox{\rm Scal}^{\tilde{g}}_{x} = 
\mbox{\rm Scal}^{\tilde{g}_0}_{x}  - 2 \alpha_x^2  \Bigl( \partial_{1}^2 h_2(x) + \partial_{1}^2 h_3(x)\Bigr) -4 \alpha_x \sum_{j=2}^n  \left(P_x\right)_{j,1} \Bigl( \partial^2_{1j} h_2(x) + \partial^2_{1j} h_3(x)\Bigr)\\ 
- 2 \alpha_x^2   \Bigl(  \left(\partial_{1} h_2(x) \right)^2 + \left(\partial_{1} h_3(x) \right)^2 +  \partial_1 h_{2}(x)  \partial_{1} h_{3}(x)\Bigr)\\
 + \Psi_2 \left(x, \check{J}^2_{h}(x)\right) \, \partial_1h_2(x) + \Psi_3 \left(x, \check{J}^2_{h}(x)\right) \, \partial_1h_3(x) + \Psi \left( x, \check{J}^2_{h}(x)\right),
\end{multline*}
where $ \Psi, \Psi_{2}, \Psi_{3}: \R^n \times \check{J}^2(\T^n, \R^2) \rightarrow \R$ are smooth functions and $\Psi (x,0)=0$ for all $x\in \R^n$. 
\end{lemma}

Set $\bar{\epsilon}:=\epsilon/2$. The desired result follows, by setting $h_2=F_1$ and $h_3=F_2$, if we show that there exists a smooth function $F=(F_1,F_2)\in C^{\infty}_{per}(\R^n, \R^2)$ such that 
\[
\left\| F\right\|_{C^0} < \bar{\epsilon} \quad \mbox{and} \quad  J^2_F(x) \in \mathcal{R}_{\bar{\epsilon}},
\]
where $\mathcal{R}\subset J^2(\T^n,\R^2)$ is the semilinear (in $x_1$) partial differential relation  defined by 
\[
\mathcal{R} := \Bigl\{  \sigma \in J^2(\T^n,\R^2) \, \vert \,  L\left(\sigma_0\right) \cdot \sigma_{11}^2 +   R\left(\sigma_0,\check{\sigma},\sigma_1^1, \sigma^2_{1*}\right)=0 \Bigr\}.
\]
with 
\[
L \left(\sigma_0\right)= -2 \alpha_{\sigma_0}^2(1,1)
\]
and
\begin{multline*}
R\left(\sigma_0, \check{\sigma},\sigma^1_1, \sigma^2_{1*}\right) =  - 2\alpha_{\sigma_0}^2 \Bigl( \left(\sigma^1_1\right)_1^2 + \left(\sigma^1_1\right)_2^2 + \left(\sigma^1_1\right)_1 \left(\sigma^1_1\right)_2\Bigr)\\
- 4\alpha_{\sigma_0} \sum_{j=2}^n  \left(P_x\right)_{j,1}  \Bigl( \left(\sigma^2_{1j}\right)_1 + \left( \sigma^2_{1j}\right)_2 \Bigr)\\
+ \Psi \left(\check{\sigma}\right) + \Psi_{2}  \left(\check{\sigma}\right)  \,  \left(\sigma^1_1\right)_1 + \Psi_3 \left(\check{\sigma}\right)  \,  \left(\sigma^1_1\right)_2 + \tilde{k}\left(\sigma_0\right) + \bar{\epsilon}.
\end{multline*}
To apply Proposition \ref{PROPmixedapp} with $f\equiv 0$, we need to construct two well-adapted ($\Z^n$-periodic in the $x$ variable) smooth families of loops $\{\gamma_x\}_{x\in \R^n}$ and $\{\delta_x\}_{x\in \R^n}$ valued in $\R^2$, such that, for every $x\in \R^n$, the following conditions hold:
\[
\overline{\gamma}_x=\overline{\delta}_x=0, \quad (\gamma_x)_1 + (\gamma_x)_2 = 0,
\]
and
\begin{multline*}
2 \alpha_{x}^2 \left((\delta_x)_1 + (\delta_x)_2 \right) =  -2 \alpha_x^2 \Bigl( \left(\left(\gamma_x\right)_1\right)^2 + \left(\left(\gamma_x\right)_2\right)^2 +\left(\gamma_x\right)_1 \left(\gamma_x\right)_2\Bigr)\\
  - 4\alpha_x \sum_{j=2}^n  \left(P_x\right)_{j,1} \Bigl( \left(\partial_j\gamma_{x}\right)_1 +  \left(\partial_j\gamma_{x}\right)_2 \Bigr) \\
 +  \Psi_{2} (x,0)  \,  \left(\gamma_x\right)_1 + \Psi_3 (x,0)  \,  \left(\gamma_x\right)_2  + \tilde{k}(x) + \bar{\epsilon},
\end{multline*}
where we have used $\Psi(x,0)=0$. From these conditions, we must have $(\gamma_x)_2 = -(\gamma_x)_1$ for all $x\in \T^n$. This implies that $(\partial_j\gamma_{x})_1 + (\partial_j \gamma_{x})_2 =0$ for all $j=2, \ldots,n$. Thus, if we further impose $(\delta_x)_2=0$, then $(\delta_x)_1$ must be defined as 
\begin{eqnarray}\label{22avril1}
2 \alpha_{x}^2 (\delta_x)_1    =  - 2 \alpha_x^2 \left(\left(\gamma_x\right)_1\right)^2 
 +  \left( \Psi_{2} (x,0)  - \Psi_3 (x,0)\right) \left(\gamma_x\right)_1  + \tilde{k}(x) + \bar{\epsilon},
\end{eqnarray}
and it must satisfy $\overline{\left(\delta_x\right)_1 } =0$. Since $\int_0^1 \cos^2 (2\pi t) \, dt=1/2$, the function  
\[
(\gamma_x)_1(t) := \frac{\sqrt{\tilde{k}(x)+\bar{\epsilon}}}{\alpha_x} \, \cos (2\pi t) \qquad \forall x  \in \T^n, \, \forall t \in \R
\]
ensures that $(\delta_x)_1$, as defined in (\ref{22avril1}), satisfies $\overline{\left(\delta_x\right)_1 } =0$. The result then follows by Proposition~\ref{PROPmixedapp} and by considering the metric on $\T^n$ obtained from $\tilde{g}$ via the quotient operation. 
\end{proof}

\subsection{Prescribing scalar curvature on thick tori}\label{ThickTori}

The following result will plays a key role in the proof of Theorem \ref{THMScalar}. Its proof is simply an adaption of the proof of Proposition \ref{PROPTnGen}.

\begin{proposition}\label{PROPThickTori}
Let $n,d$ be two integers with $n\geq 3$ and $1 \leq d \leq n-1$. Let $\delta>0$, and let 
\[
\Phi : (-\delta,1+\delta) \times B^{n-d-1}(0,\delta) \times \T^d  \longrightarrow V := \mbox{\rm Im}(\Phi)  \subset \R^n,
\]
be a smooth diffeomorphism. Let $h, h_0$ be two smooth Riemannian metrics on the open set $V\subset \R^n$. Then, for any $\nu>0$, any smooth function $s: V \rightarrow [0,\infty)$, and any compact subset $C$ of $V$, there exists a smooth Riemannian metric $h_{\nu}$ on $V$ satisfying the following properties:
\begin{itemize}
\item[(i)] $\| h_{\nu}-h_0\|_{C^0_h} < \nu$ on $C$.
\item[(ii)] For every $z\in C$, $-s(z)^2 - \nu < \mbox{\rm Scal}^{h_{\nu}}_z - \mbox{\rm Scal}^{h_0}_z < -s(z)^2 + \nu$.
\item[(iii)] For every $z\in V$, $s(z)=0 \Longrightarrow (h_{\nu})_{z}= (h_0)_{z}$.
\end{itemize}
\end{proposition}

\begin{proof}[Proof of Proposition \ref{PROPThickTori}]
Let $\bar{h}, \bar{h}_0$ be the Riemannian metrics on 
\[
\bar{V}:=(-\delta,1+\delta) \times B^{n-d-1}(0,\delta) \times \T^d
\]
defined as the pullbacks of $h$ and $h_0$ by $\Phi$, respectively. Those metrics, along with the function $s$, can be lifted to metrics $\tilde{h}$, $\tilde{h}_0$, and a function $\tilde{s}$, on 
\[
\tilde{V}:=(-\delta,1+\delta) \times B^{n-d-1}(0,\delta)\times \R^d,
\]
with coordinates, $(x_1, \cdots, x_n)$, which are $\Z^d$-periodic in the last $d$ variables. Consider a smooth orthonormal family (with respect to $\tilde{h}_0$) of vector fields $X^1, \ldots, X^n$ on $\tilde{V}$ which is $\Z^d$-periodic in the last $d$ variables and satisfies
\begin{eqnarray*}
\alpha_x:=\langle X^1(x),e_1\rangle  >0 \quad \mbox{and} \quad \langle X^2(x),e_1\rangle= \cdots = \langle X^n(x),e_1\rangle=0 \qquad \forall x \in \tilde{V},
\end{eqnarray*}
where $e_1, \ldots, e_n$ and $\langle \cdot, \cdot \rangle$ stand respectively for the canonical basis and the canonical scalar product in $\R^n$. As in the proof of Proposition \ref{PROPTnGen},  for every $x\in \tilde{V}$ we denote by $P_x=((P_x)_{ij})$ the $n\times n$ matrix with columns $X^1(x), \ldots, X^n(x)$ and we define a perturbed metric $\hat{h}$ on $\tilde{V}$ by 
\begin{eqnarray}\label{gtildeformula}
\hat{h}_x = Q_x^{\mathsf{tr}}D(x)Q_x \quad \mbox{with} \quad Q_x:=P_x^{-1}   \qquad \forall x \in \tilde{V},
\end{eqnarray}
where $D$ is a diagonal matrix with coefficients $(1, e^{2h_2}, e^{2h_3},1, \ldots,1)$ where $h_2, h_3 : \tilde{V} \rightarrow \R$ are two smooth functions that are $\Z^d$-periodic in the last $d$ variables. Let $\tilde{C}$ be the lift of $\bar{C}:=\Phi^{-1}(C)\subset \bar{V}$ in $\R^n$. The desired result will follow, by setting $h_2(x)=F_1(x)$ and $h_3(x)=F_2(x)$ for all $x\in E:= \tilde{C} \subset \tilde{V}$, if we show that, for any $\nu>0$ there is a smooth function $F=(F_1,F_2) \in C^{\infty}_{per,d}(E,\R^m)$  such that 
\[
\left\| F\right\|_{C^0} < \nu \quad \mbox{and} \quad  J^2_F(x) \in \mathcal{R}_{\nu},
\]
where $\mathcal{R}\subset J^2(E,\R^2)$ is the semilinear partial differential relation  defined by 
\[
\mathcal{R} := \Bigl\{  \sigma \in J^2(E,\R^2) \, \vert \,  L\left(\sigma_0\right)\cdot \sigma_{11}^2 +   R\left(\sigma_0,\check{\sigma},\sigma_1^1, \sigma^2_{1*}\right)=0 \Bigr\}.
\]
with 
\[
L  \left(\sigma_0\right) = -2 \alpha_{\sigma_0}^2 (1,1)
\]
and
\begin{multline*}
R\left( \sigma_0,\check{\sigma},\sigma^1_1, \sigma^2_{1*}\right) =  - 2 \alpha_{\sigma_0}^2 \Bigl( \left(\sigma^1_1\right)_1^2 + \left(\sigma^1_1\right)_2^2 + \left(\sigma^1_1\right)_1 \left(\sigma^1_1\right)_2\Bigr) \\
- 4\alpha_{\sigma_0} \sum_{j=2}^n  \left(P_x\right)_{j,1}  \Bigl( \left(\sigma^2_{1j}\right)_1 + \left( \sigma^2_{1j}\right)_2 \Bigr)\\
+ \Psi \left(\check{\sigma}\right) + \Psi_{2}  \left(\check{\sigma}\right)  \,  \left(\sigma^1_1\right)_1 + \Psi_3 \left(\check{\sigma}\right)  \,  \left(\sigma^1_1\right)_2 + \tilde{s}(\sigma_0)^2,
\end{multline*}
where the maps $\Psi$ and $\Psi_i$ come from Lemma~\ref{PROPTnGenLEM}. To apply Proposition \ref{PROPmixedapp} with $f\equiv 0$, we need to construct two well-adapted smooth families of loops $\{\gamma_x\}_{x\in E}$ and $\{\delta_x\}_{x\in E}$ valued in $\R^2$, such that, for every $x\in E$, the following conditions hold:
\[
\overline{\gamma}_x=\overline{\delta}_x=0, \quad (\gamma_x)_1 + (\gamma_x)_2 = 0,
\]
and
\begin{multline*}
2 \alpha_{\sigma_0}^2\left( (\delta_x)_1 + (\delta_x)_2\right)  =  - 2 \alpha_{\sigma_0}^2 \Bigl( \left(\left(\gamma_x\right)_1\right)^2 + \left(\left(\gamma_x\right)_2\right)^2 +\left(\gamma_x\right)_1 \left(\gamma_x\right)_2\Bigr) \\
- 4 \alpha_{\sigma_0} \sum_{j=2}^n  \left(P_x\right)_{j,1} \Bigl( \left(\partial_j\gamma_{x}\right)_1 +  \left(\partial_j\gamma_{x}\right)_2 \Bigr) \\
 +  \Psi_{2} (x,0)  \,  \left(\gamma_x\right)_1 + \Psi_3 (x,0)  \,  \left(\gamma_x\right)_2  + \tilde{s}(x)^2,
\end{multline*}
where we have used $\Psi(x,0)=0$. From these conditions, we must have $(\gamma_x)_2 = -(\gamma_x)_1$ for all $x\in E$. This implies that $(\partial_j\gamma_{x})_1 + (\partial_j \gamma_{x})_2 =0$ for all $j=2, \ldots,n$. Thus, if we further impose $(\delta_x)_2=0$, then $(\delta_x)_1$ must be defined as 
\begin{eqnarray}\label{deltax1thick}
2 \alpha_{x}^2 (\delta_x)_1  =  - 2\alpha_x^2 \left(\left(\gamma_x\right)_1\right)^2 
 +  \Psi_{2} (x,0)  \,  \left(\gamma_x\right)_1 + \Psi_3 (x,0)  \,  \left(\gamma_x\right)_2  + \tilde{s}(x)^2,
\end{eqnarray}
and it must satisfy $\overline{\left(\delta_x\right)_1 } =0$. Since $\int_0^1 \cos^2 (2\pi t) \, dt=1/2$, the function  
\[
(\gamma_x)_1(t) := \frac{\tilde{s}(x)}{\alpha_x} \, \cos (2\pi t) \qquad \forall x  \in E, \, \forall t \in \R
\]
ensures that $(\delta_x)_1$, as defined in (\ref{deltax1thick}), satisfies $\overline{\left(\delta_x\right)_1 } =0$. By construction, if $\tilde{s}(x)=0$ then $\delta_x=\gamma_x\equiv 0$ and as a consequence $\hat{h}_x=(\tilde{h}_0)_x$. Therefore, by Proposition \ref{PROPmixedapp}, the above discussion shows that for every $\nu>0$, there exists a metric $\tilde{h}_{\nu}$ on $\tilde{V}$, that is periodic in the last $d$ variables, such that 
\begin{eqnarray*}
\bigl\| \tilde{h}_{\nu}-\tilde{h}_0\bigr\|_{C^0_{\tilde{h}}} < \nu,
\end{eqnarray*}
where the $C^0$-norm is taken over $E$, 
\begin{eqnarray*}
-\tilde{s}(x)^2 - \nu < \mbox{\rm Scal}^{\tilde{h}_{\nu}}_x - \mbox{\rm Scal}^{\tilde{h}_0}_x < -\tilde{s}(x)^2 + \nu \qquad \forall x \in E,
\end{eqnarray*}
and
\begin{eqnarray*}
\tilde{s}(x) = 0 \quad \Longrightarrow \quad (\tilde{h}_{\nu})_x=(\tilde{h}_0)_x \qquad \forall x \in \tilde{V}.
\end{eqnarray*}
The result then follows by Proposition \ref{PROPmixedapp} and by considering the metric $h_{\nu}$ obtained by pushforward of $\tilde{h}_{\nu}$ by the quotient map and the map $\Phi$.
\end{proof}

\begin{remark}\label{REMPoint}
The above proof easily demonstrate that a Riemannian metric can be perturbed into a Riemannian metric with controlled scalar curvature at a given point. More specifically, let $h$ and $h_0$ be two smooth Riemannian metrics on an open neighborhood $V$ of the origin in $\R^n$. Then, for any $\nu>0$, any smooth function $s: V \rightarrow [0,\infty)$, and any compact subset $C$ of $V$, there exists a smooth Riemannian metric $h_{\nu}$ on $V$ satisfying the following properties:
\begin{itemize}
\item[(i)] $\| h_{\nu}-h_0\|_{C^0_h} < \epsilon$ on $C$.
\item[(ii)] For every $z\in C$, $-s(z)^2 - \nu < \mbox{\rm Scal}^{h_{\nu}}_z - \mbox{\rm Scal}^{h_0}_z < -s(z)^2 + \nu$.
\item[(iii)] For every $z\in V$, $s(z)=0 \Longrightarrow (h_{\nu})_{z}= (h_0)_{z}$.
\end{itemize}
Of course, by considering a suitable chart, this result can be extended to manifolds.  
\end{remark}

\section{Proof of Theorem \ref{THMScalar}}\label{SEC4}

The proof of Theorem \ref{THMScalar} consists of equipping the manifold $M$ with a triangulation and solving the relation simplex by simplex, starting from the $0$-dimensional simplices and proceeding up to the $n$-dimensional simplices. Before presenting the proof, we recall some basic facts about simplicial complexes and triangulations, referring the reader to \cite{munkres66} for further details.\\ 

Let $k$ a positive integer be fixed. A simplex of dimension $d$ in $\R^k$, with $0\leq d \leq k$, is a subset $\sigma$ of $\R^k$ defined as the convex hull of $d+1$ affinely independent points, called its vertices. More precisely, these are $(d+1)$ points $v_0, \ldots, v_d\in \R^k$ such that $v_1-v_0, \ldots, v_d-v_0$ are linearly independent. A face of a simplex $\sigma$ is any simplex obtained by taking the convex hull of a subset of the vertices of $\sigma$. A simplicial complex $\mathcal{K}$ in $\R^k$ is a collection of simplices in $\R^k$ that satisfy the following natural compatibility conditions: 
\begin{itemize}
\item[(1)] Every face of a simplex in $\mathcal{K}$ is itself an element of $\mathcal{K}$.
\item[(2)] For any two simplices $\sigma_1, \sigma_2 \in \mathcal{K}$, if their intersection $\sigma_1 \cap \sigma_2$ is non-empty then, it is a face of both $\sigma_1$ and $\sigma_2$.
\item[(3)] The set 
\[
|\mathcal{K}|:=\bigcup_{\sigma\in \mathcal{K}}\sigma,
\] 
called the geometric realization of $\mathcal{K}$, has the property that every point in it has a neighborhood that intersects only finitely many simplices in $\mathcal{K}$.
\end{itemize}
The dimension of the simplicial complex $\mathcal{K}$ is defined as the maximum of the dimensions of its simplices. A smooth triangulation of a smooth manifold $N$ is a pair $(\mathcal{K},\varphi_{\mathcal{K}})$, where $\mathcal{K}$ is a simplical complex in some Euclidean space, and $\varphi_{\mathcal{K}}: |\mathcal{K}| \rightarrow N$ is a homeomorphism with the additional property that the restriction of $\varphi_{\mathcal{K}}$ to each simplex in $\mathcal{K}$ is a smooth embedding. The celebrated Whitehead Theorem asserts that every smooth manifold admits a smooth triangulation. Moreover, for a smooth manifold with boundary, any triangulation of the boundary can be extended to a triangulation of the entire manifold.\\

To prove Theorem \ref{THMScalar} we begin by recalling that $M$ is equipped with a smooth Riemannian metric $g$. Additionally, we consider a smooth Riemannian metric $g_0$ on $M$, a smooth function $k: M \rightarrow (0,+\infty)$, and $\epsilon >0$. Since the smooth manifold $M$ admits a smooth triangulation, Theorem \ref{THMScalar} will follow from the following property for $d=n$:\\

\noindent $(P_d)$ Let $\mathcal{K}$ be a simplicial complex of dimension $d$, and let $\varphi_{\mathcal{K}}: |\mathcal{K}| \rightarrow M$ be a continuous map that is a homeomorphism onto its image, with the property that its restriction to each simplex of $\mathcal{K}$ is a smooth embedding and the image of each simplex lies within a local chart of $M$. Let $\mathcal{L}$ be a simplicial subcomplex of $\mathcal{K}$, meaning it is a simplicial complex whose simplices are contained within $\mathcal{K}$, which may be empty. Let  $g_{\mathcal{L}}$ be a smooth Riemannian metric on $M$, $U_{\mathcal{L}}$ be an open subset of $M$ containing $\varphi_{\mathcal{K}}(|\mathcal{L}|)$, and let $\alpha \in (0,1)$ such that the following properties are satisfied:
\begin{itemize}
\item[$(P_d1)$] $\| g_{\mathcal{L}} -g_0\|_{C^0_g} < \epsilon$ on $M$.
\item[$(P_d2)$] $ \mbox{\rm Scal}^{g_{\mathcal{L}}}<\mbox{\rm Scal}^{g_{0}}-k- \alpha  \epsilon$ on $U_{\mathcal{L}}$.
\item[$(P_d3)$] $\mbox{\rm Scal}^{g_{\mathcal{L}}}> \mbox{\rm Scal}^{g_{0}}-k - \epsilon $ on $M$.
\end{itemize}
Then, for every $\beta \in (0,\alpha)$, there exists a smooth Riemannian metric $g_{\mathcal{K}}$ on $M$ that satisfies: 
\begin{itemize}
\item[$(P_d4)$] $\| g_{\mathcal{K}} -g_0\|_{C^0_g} < \epsilon$ on $M$.
\item[$(P_d5)$] $ \mbox{\rm Scal}^{g_{\mathcal{K}}}<\mbox{\rm Scal}^{g_{0}}-k- \beta  \epsilon$ on $\varphi_{\mathcal{K}}(|\mathcal{K}|)$.
\item[$(P_d6)$] $\mbox{\rm Scal}^{g_{\mathcal{K}}}> \mbox{\rm Scal}^{g_{0}}-k - \epsilon $ on $M$.
\item[$(P_d7)$] $g_{\mathcal{K}} = g_{\mathcal{L}}$ on $U_{\mathcal{L}}$.
\end{itemize}
\bigskip

For $d=n$, a triangulation $(\mathcal{K},\varphi_{\mathcal{K}})$ of $M$, $\mathcal{L}=\emptyset$, and $\alpha=1/2$, property $(P_d)$ provides a metric $g_{\epsilon}=g_{\mathcal{K}}$ satisfying the conclusion of Theorem \ref{THMScalar}. We now proceed to prove $(P_d)$ by induction for every integer $d \in [0,n]$. \\

The property $(P_0)$ follows from the fact that the geometric realization of a simplicial complex of dimension $0$ is a locally finite union of points and Remark \ref{REMPoint}.\\

Assume that $(P_{d'})$ holds for any $d'\in \{0, \ldots, d\}$ with $d\in \{0, \ldots, n-1\}$ and let us prove $(P_{d+1})$. Let $\mathcal{K}$ be a simplicial complex of dimension $d+1$, and let $\varphi_{\mathcal{K}}: |\mathcal{K}| \rightarrow M$ be a continuous map that is a homeomorphism onto its image, with the property that its restriction to each simplex of $\mathcal{K}$ is a smooth embedding and the image of each simplex lies within a local chart of $M$. Let $\mathcal{L}$ be a simplicial subcomplex of $\mathcal{K}$, let $g_{\mathcal{L}}$ be a smooth Riemannian metric on $M$, and let $U_{\mathcal{L}}$ be an open subset of $M$ containing $\varphi_{\mathcal{K}}(|\mathcal{L}|)$. Fix $\alpha, \beta \in (0,1)$ such that $\beta<\alpha$, and assume that $(P_d1)$-$(P_d3)$ hold. Define $\mathcal{M}$ to be the simplicial complex consisting of all simplices of $\mathcal{K}$ of dimension less than or equal to $d$, together with all simplices of $\mathcal{L}$. We claim that, by the property $(P_d)$ given by the induction hypothesis, there exists a smooth Riemannian metric $g_{\mathcal{M}}$ on $M$ satisfying the following properties:
\begin{itemize}
\item[(a)] $\| g_{\mathcal{M}} -g_0\|_{C^0_g} < \epsilon$ on $M$.
\item[(b)] $ \mbox{\rm Scal}^{g_{\mathcal{M}}}<\mbox{\rm Scal}^{g_{0}}-k-(\beta + \alpha)\epsilon/2$ on $\varphi_{\mathcal{K}}(|\mathcal{M}|)$.
\item[(c)] $\mbox{\rm Scal}^{g_{\mathcal{M}}}> \mbox{\rm Scal}^{g_{0}}-k - \epsilon $ on $M$.
\item[(d)] $g_{\mathcal{M}} = g_{\mathcal{L}}$ on $U_{\mathcal{L}}$. 
\end{itemize}
To establish this, let $\mathcal{L}_d$ and $\mathcal{M}_d$ denote the simplicial complexes consisting of all simplices of $\mathcal{L}$ and $\mathcal{M}$, respectively, that have  dimension less than or equal to $d$. By property $(P_{d})$, since $\mathcal{M}_d$ has dimension $d$, $\mathcal{L}_d \subset \mathcal{M}_d,$ $\varphi_{\mathcal{K}}(\mathcal{L}_d) \subset \varphi_{\mathcal{K}}(\mathcal{L}) \subset U_{\mathcal{L}}$, and $(\beta+\alpha)/2<\alpha$, there exists a smooth Riemannian metric $g_{\mathcal{M}_d}$ on $M$ satisfying
\[
\left\| g_{\mathcal{M}_d} -g_0\right\|_{C^0_g} < \epsilon, \quad \mbox{\rm Scal}^{g_{\mathcal{M}_d}}<\mbox{\rm Scal}^{g_{0}}-k -\frac{(\beta+\alpha)\epsilon}{2}\, \mbox{ on } \, \varphi_{\mathcal{K}}(|\mathcal{M}_d|),
\]
\[
\mbox{\rm Scal}^{g_{\mathcal{M}_d}} > \mbox{\rm Scal}^{g_{0}}-k - \epsilon  \, \mbox{ on } \, M, \quad \mbox{and} \quad g_{\mathcal{M}_d} = g_{\mathcal{L}}  \, \mbox{ on } \, \mathcal{U}_{\mathcal{L}}.
\] 
Since $\varphi_{\mathcal{K}}(|\mathcal{M}|)=\varphi_{\mathcal{K}}(|\mathcal{M}_d|)\cup \varphi_{\mathcal{K}}(|\mathcal{L}|)$, $(P_d2)$ holds, and $(\beta+\alpha)/2<\alpha$,  the metric  $g_{\mathcal{M}}=g_{\mathcal{M}_d}$  satisfies the required properties (a)-(d).   

It remains to achieve the construction of $g_{\mathcal{K}}$ on the images by $\varphi_{\mathcal{K}}$ of simplices of $\mathcal{K}$ which are not in $\mathcal{M}$. These simplices have dimension $d+1$, and their boundary is contained in $\varphi_{\mathcal{K}}(|\mathcal{M}|)$. We need the following lemma whose proof is postponed to Appendix \ref{AppProofLEMToricDiscs}:

\begin{lemma}\label{LEMToricDiscs}
For every integer $k\in [2,n]$, there exists a compact set $\Sigma^k \subset D^k$, the closed unit disc in $\R^k$, with the following properties:
\begin{itemize}
\item[(i)] There exists an open smooth submanifold $\tilde{\Sigma}^k$ of $\R^k$ with codimension $2$ such that $\Sigma^k=\tilde{\Sigma}^k \cap D^k$, and $\tilde{\Sigma}^k$ is transverse to $\partial D^k$. 
\item[(ii)] For every open neighborhood $\mathcal{N}$ of $\Sigma^k\cup \partial D^k$, there exists a smooth map $\Phi: \T^{k-1} \times [0,1] \rightarrow \mbox{Int}(D^k\setminus \Sigma^k)$ such that $\Phi$ is a diffeomorphism onto its image and $\Phi( \T^{k-1} \times \{0,1\})\subset \mathcal{N}$. 
\end{itemize}
\end{lemma}

\begin{figure}[H]
\begin{center}
\includegraphics[width=4cm]{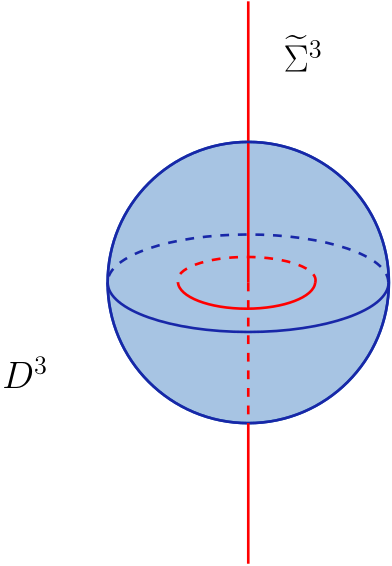}
\caption{The codimension $2$ manifold $\tilde{\Sigma}^3$ associated with $D^3$, shown in red, is the union of a vertical segment containing the origin and a planar circle centered at the origin\label{fig2}}
\end{center}
\end{figure}

Let us describe how to construct $g_{\mathcal{K}}$ on $\varphi_{\mathcal{K}}(\sigma)$, where $\sigma$ is a $(d+1)$-dimensional simplex in $\mathcal{K}$ that does not belong to $\mathcal{M}$. Since all faces of $\sigma$ with dimension at most $d$ are contained in $\mathcal{M}_d$, property (b) above provides control over the scalar curvature of $g_{\mathcal{M}}$ on the topological boundary $\partial \varphi_{\mathcal{K}}(\sigma)$ of $\varphi_{\mathcal{K}}(\sigma)$ in $M$. Consequently, there exists an open set $U\subset M$ containing $\partial \varphi_{\mathcal{K}}(\sigma)$ such that
\[
U_{\mathcal{L}} \cup \partial \varphi_{\mathcal{K}}(\sigma) \subset U \quad \mbox{and} \quad \mbox{\rm Scal}^{g_{\mathcal{M}}}<\mbox{\rm Scal}^{g_{0}}-k - \frac{(\beta+\alpha)\epsilon}{2} \, \mbox{ on } \, U.
\]
Since $\varphi_{\mathcal{K}}(\sigma)$ lies within a local chart of $M$, we assume, without loss of generality, that we are working in $\R^n$. That is, we assume the existence of an open set $W\subset \R^n$ containing both $U$ and $\varphi_{\mathcal{K}}(\sigma)$, and our goal is to construct a smooth Riemannian metric $g_{\mathcal{K}}$ on $W$ satisfying the following conditions: 
\begin{itemize}
\item[(e)] $\| g_{\mathcal{K}} -g_0\|_{C^0_g} < \epsilon$ on $W$.
\item[(f)] $ \mbox{\rm Scal}^{g_{\mathcal{K}}}<\mbox{\rm Scal}^{g_{0}}-k- \beta  \epsilon$ on $\varphi_{\mathcal{K}}(\sigma)$.
\item[(g)] $\mbox{\rm Scal}^{g_{\mathcal{K}}}> \mbox{\rm Scal}^{g_{0}}-k - \epsilon $ on $W$.
\item[(h)] $g_{\mathcal{K}} = g_{\mathcal{L}}$ on $U_{\mathcal{L}}$.
\end{itemize}
Consider a smooth diffeomorphism $F:\R^{d+1} \rightarrow \R^{d+1}$ that maps the $(d+1)$-dimensional closed unit disc $D^{d+1}$ to a set $D$ satisfying the following properties: 
\[
D \subset \mbox{Int} \left(\varphi_{\mathcal{K}}(\sigma)\right), \quad \varphi_{\mathcal{K}}(\sigma) \subset U \cup D, \quad \mbox{and} \quad \partial D \subset U.
\] 
Next, apply Lemma \ref{LEMToricDiscs} with $k=d+1$ and define
\[
\Sigma := F \left(\Sigma^{d+1}\right).
\]
By construction, specifically by property (i) of Lemma \ref{LEMToricDiscs}, $\Sigma$ is a smooth submanifold with boundary of dimension $d-1$ in $D=F(D^{d+1})$, and its boundary  $\partial \Sigma$ is contained within the boundary $\partial D$ of $D$. By Whitehead's Theorem, $\Sigma$ admits a triangulation $(\mathcal{S},\varphi_{\mathcal{S}})$ induced by a triangulation of $\partial \Sigma$. In other words, there exists a simplicial subcomplex $\mathcal{P}$ of $\mathcal{S}$ such that $\varphi_{\mathcal{S}} (|\mathcal{P}|)=\partial \Sigma$. Since $\partial \Sigma \subset \partial D \subset U$, property $(P_{d-1})$ from the induction hypothesis (applied with $\mathcal{K}=\mathcal{S}$, $\varphi_{\mathcal{K}}=\varphi_S$, $\mathcal{L}=\mathcal{P}$, $g_{\mathcal{L}}=g_{\mathcal{M}}$, $U_{\mathcal{L}}=U$, $\alpha=(\beta+\alpha)/2$, and $\beta=(3\beta+\alpha)/4$) yields a metric $g_{\mathcal{S}}$ on $W$ satisfying
\[
\left\| g_{\mathcal{S}} -g_0\right\|_{C^0_g} < \epsilon, \quad  \mbox{\rm Scal}^{g_{\mathcal{S}}}<\mbox{\rm Scal}^{g_{0}}-k - \frac{(3\beta+\alpha)\epsilon}{4} \, \mbox{ on } \, \varphi_{\mathcal{S}}(|\mathcal{S}|),
\]
\[
\mbox{\rm Scal}^{g_{\mathcal{S}}}> \mbox{\rm Scal}^{g_{0}}-k - \epsilon  \, \mbox{ on } \, M, \quad \mbox{and} \quad g_{\mathcal{S}} = g_{\mathcal{M}}  \, \mbox{ on } \, U.
\] 
Then, define the smooth Riemannian metric $g_{\mathcal{M}}'$ on $W$ by 
\[
\left\{
\begin{array}{rcl}
g_{\mathcal{M}}' & = & g_{\mathcal{S}} \mbox{ on } \varphi_{\mathcal{K}} (\sigma)\\
g_{\mathcal{M}}' & = & g_{\mathcal{M}} \mbox{ on } W \setminus \varphi_{\mathcal{K}} (\sigma).
\end{array}
\right.
\]
Since $g_{\mathcal{S}}$ coincides with $g_{\mathcal{M}}$ on the open set $U$ that contains $\partial \varphi_{\mathcal{K}} (\sigma)$, $g_{\mathcal{M}}'$ is smooth, and moreover,  the control of its scalar curvature on $\varphi_{\mathcal{S}}(|\mathcal{S}|)=\Sigma$ can be extended to an open set $U'$ containing $U$ and $\Sigma$. Therefore, $g_{\mathcal{M}}'$ satisfies:
\begin{itemize}
\item[(i)] $\| g_{\mathcal{M}}' -g_0\|_{C^0_g} < \epsilon$ on $W$.
\item[(j)] $\mbox{\rm Scal}^{g_{\mathcal{M}}'}<\mbox{\rm Scal}^{g_{0}}-k -(3\beta+\alpha)\epsilon/4$ on $U'$.
\item[(k)] $\mbox{\rm Scal}^{g_{\mathcal{M}}'}> \mbox{\rm Scal}^{g_{0}}-k - \epsilon $ on $W$.
\item[(l)] $g_{\mathcal{M}}' = g_{\mathcal{L}}$ on $U_{\mathcal{L}}\subset U \subset U'$. 
\end{itemize}
By assertion (ii) of Lemma \ref{LEMToricDiscs} and using the diffeomorphism $F$, there exists a smooth map 
\[
\Psi : \T^d \times [0,1] \longrightarrow \mbox{Int}(D\setminus \Sigma) \subset D \subset \varphi_{\mathcal{K}}(\sigma) \subset W \subset \R^n
\]
which is a diffeomorphism onto its image and satisfies $\Phi( \T^{d-1} \times \{0,1\})\subset U'$. This immersion $\Psi$ can be smoothly extended, for some $\delta >0$, to a diffeomorphism 
\[
\Phi : (-\delta,1+\delta) \times B^{n-d-1}(0,\delta) \times \T^d  \longrightarrow V := \Phi \left( (-\delta,1+\delta)  \times B^{n-d-1}(0,\delta)  \times \T^d \right) \subset W
\]
satisfying 
\[
T:= \Phi \left([0,1]\times \{0\}\times \T^d\right) = \Psi\left(\T^d \times [0,1]\right).
\]

\begin{figure}[H]
\begin{center}
\includegraphics[width=5cm]{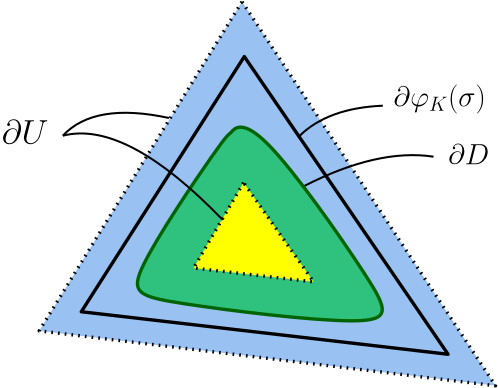}
\caption{The union of $U$ (in blue) and $D$ (in yellow) covers the simplex $\varphi_{\mathcal{K}}(\sigma)$. The green region corresponds to points belonging to both $U$ and $D$\label{fig2}}
\end{center}
\end{figure}

Fix a compact neighborhood $C\subset V$ of $T$. By properties (i) and (k), there exists 
\begin{eqnarray}\label{1avrileq1}
\nu\in \left(0,\min \left\{ \epsilon,\frac{(\alpha-\beta)\epsilon}{8}\right\}\right)
\end{eqnarray}
such that
\begin{eqnarray}\label{1avrileq3}
\| g_{\mathcal{M}}' -g_0\|_{C^0_g}+\nu < \epsilon
\end{eqnarray}
and
\begin{eqnarray}\label{1avrileq2}
\mbox{\rm Scal}^{g_{\mathcal{M}}'}> \mbox{\rm Scal}^{g_{0}}-k - \epsilon +2\nu \, \mbox{ on } \, C.
\end{eqnarray}

\begin{figure}[H]
\begin{center}
\includegraphics[width=5cm]{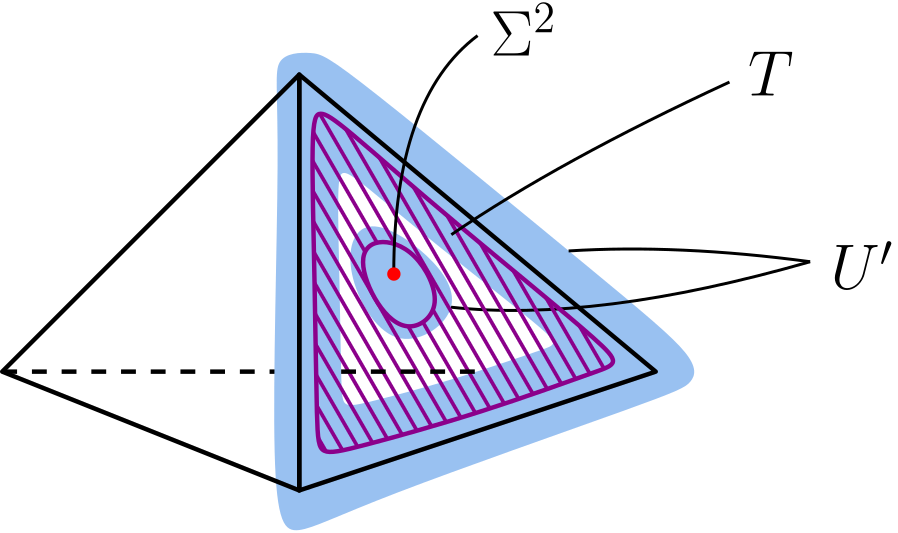}
\caption{The sets $U'$, $\Sigma^2$, and $T$ in the case where $\varphi_{\mathcal{K}}(\sigma)$ is a $2$-dimensional simplex in ambient dimension $3$\label{fig4}}
\end{center}
\end{figure}

Next, consider a smooth function $s:V \rightarrow [0,\infty)$  satisfying the following conditions: 
\begin{itemize}
\item[(m)] $s=0$ on $U' \cup (V\setminus C)$. 
\item[(n)] $s^2 - \nu > \mbox{\rm Scal}^{g_{\mathcal{M}}'} - \mbox{\rm Scal}^{g_0} + k + \beta \epsilon$ on $C$.
\item[(o)] $s^2 + \nu < \mbox{\rm Scal}^{g_{\mathcal{M}}'} - \mbox{\rm Scal}^{g_0} + k + \epsilon$ on $C$.
\end{itemize}
Such a function exists because 
\begin{multline*}
\mbox{\rm Scal}^{g_{\mathcal{M}}'} - \mbox{\rm Scal}^{g_0} + k + \beta \epsilon + \nu \\
= \mbox{\rm Scal}^{g_{\mathcal{M}}'} - \mbox{\rm Scal}^{g_0} + k +  \frac{(3\beta+\alpha)\epsilon}{4} + \nu -  \frac{(\alpha-\beta)\epsilon}{4} < \nu -  \frac{(\alpha-\beta)\epsilon}{4}  < -  \frac{(\alpha-\beta)\epsilon}{8} \, \mbox{ on } \, U'
\end{multline*}
by (j) and (\ref{1avrileq1}),  
\[
\mbox{\rm Scal}^{g_{\mathcal{M}}'} - \mbox{\rm Scal}^{g_0} + k + \epsilon - \nu > \nu  \, \mbox{ on } \, C
\]
from (\ref{1avrileq1})-(\ref{1avrileq2}), and , since $2\nu < (\alpha-\beta)\epsilon/4<(1-\beta)\epsilon/2$ (using (\ref{1avrileq1})), 
\[
\mbox{\rm Scal}^{g_{\mathcal{M}}'} - \mbox{\rm Scal}^{g_0} + k + \beta \epsilon + \nu < \mbox{\rm Scal}^{g_{\mathcal{M}}'} - \mbox{\rm Scal}^{g_0} + k +  \epsilon - \nu \, \mbox{ on } \, V.
\]
By applying Proposition \ref{PROPThickTori} with $h=g$, $h_0=g_{\mathcal{M}}'$, $\nu$, $s$, and $C$, we obtain a metric $h_{\nu}$ on $V$ satisfying:
\begin{itemize}
\item[(p)] $\| h_{\nu}-g_{\mathcal{M}}'\|_{C^0_h} < \nu$ on $C$.
\item[(q)] For every $z\in C$, $-s(z)^2 - \nu < \mbox{\rm Scal}^{h_{\nu}}_z - \mbox{\rm Scal}^{g_{\mathcal{M}}'}_z < -s(z)^2 + \nu$.
\item[(r)] For every $z\in V$, $s(z)=0 \Longrightarrow (h_{\nu})_{z}= (g_{\mathcal{M}}')_{z}$.
\end{itemize}
As a consequence, combining (q) with (n)-(o), we deduce
\begin{multline}\label{eq1avril4}
 \mbox{\rm Scal}^{g_0}  -k - \epsilon <  \mbox{\rm Scal}^{g_{\mathcal{M}}'}  -s^2 - \nu < \mbox{\rm Scal}^{h_{\nu}} \\
 < \mbox{\rm Scal}^{g_{\mathcal{M}}'}  -s^2 + \nu < \mbox{\rm Scal}^{g_0}  -k - \beta \epsilon   \, \mbox{ on } \, C.
\end{multline}
Moreover, since $h_{\nu}$ coincides with $g_{\mathcal{M}}'$ on the open set $V\setminus C$ by (m) and (r),  these metrics have the same scalar curvature on $V\setminus C$. Using $(k)$, we obtain
\begin{eqnarray}\label{eq1avril5}
\mbox{\rm Scal}^{h_{\nu}} = \mbox{\rm Scal}^{g_{\mathcal{M}}'}  > \mbox{\rm Scal}^{g_0}  -k - \epsilon   \, \mbox{ on } \, V\setminus C.
\end{eqnarray}
We conclude by defining $g_{\mathcal{K}}$ on $W$ as follows:
\[
\left\{
\begin{array}{rcl}
g_{\mathcal{K}} & = & h_{\nu} \mbox{ on } V\\
g_{\mathcal{K}} & = & g_{\mathcal{M}}' \mbox{ on } W\setminus V.
\end{array}
\right.
\]
By construction, $g_{\mathcal{K}}$ is smooth, property (e) follows from (\ref{1avrileq3}) and (p), (f) follows from (\ref{eq1avril4}) and $T\subset C$, (g) follows from (\ref{eq1avril4})-(\ref{eq1avril5}) and (k), and finally (h) follows from  (l) and (m).

\appendix
\section{Ricci and scalar curvatures in Riemannian geometry}\label{APPRiemFormulas}

The aim of this section is to recall the definitions of Ricci and scalar curvatures and to provide formulas to compute them in local coordinates. We adopt the notations of the textbook by Gallot, Hulin and Lafontaine \cite{ghl04}.

\subsection{Main definitions}\label{APP1DEF}

We condider a smooth compact manifold $M$ of dimension $n\geq 2$ equipped with a smooth Riemannian metric $g$. The Levi-Civita connection $\nabla^g$ of the metric $g$ is the
unique connection which is torsion-free and consistent with $g$
where the first property means that ($\Gamma(TM)$ stands for the set
of smooth vector fields on $M$)
\begin{eqnarray*}
\nabla^g_XY-\nabla^g_YX=[X,Y] \qquad \forall X,Y\in \Gamma(TM)
\end{eqnarray*}
and the latter
\begin{eqnarray*}
X\cdot g(Y,Z)=g\left(\nabla^g_XY,Z\right) +
g\left(Y,\nabla^g_XZ\right)\qquad \forall X,Y,Z\in \Gamma(TM).
\end{eqnarray*}
The Riemann curvature endomorphism $R^g:\Gamma(TM)\times \Gamma(TM) \times
\Gamma(TM) \rightarrow \Gamma(TM)$ of $g$ is the $(1,3)$-tensor
defined by
\begin{eqnarray*}
R^g(X,Y)Z :=
\nabla^g_{Y}\left(\nabla^g_XZ\right)-\nabla^g_{X}\left(\nabla^g_YZ\right)+\nabla^g_{[X,Y]}Z
\qquad \forall X,Y,Z \in \Gamma(TM)
\end{eqnarray*}
and the Riemann curvature tensor $R^g:\Gamma(TM)\times \Gamma(TM) \times \Gamma(TM) \times
\Gamma(TM) \rightarrow \R$  of $g$, also denoted by $R^g$,  is the $(0,4)$-tensor given by
\begin{eqnarray*}
R^g(X,Y,Z,T) := g\left(R^g(X,Y)Z, T \right) \qquad \forall X,Y,Z, T \in \Gamma(TM).
\end{eqnarray*}
The Ricci curvature is the $(0,2)$-tensor defined as the trace with respect to $g$ of the endomorphism $v \in T_mM\mapsto R^g(x,v)z\in T_mM$ (with $x,z\in T_mM$) and the scalar curvature is given by the trace of the Ricci curvature with respect to the metric $g$. In other words, if  $(e_i)_{i=1,\ldots,n}$ is an orthonomal basis with respect to $g_m$ at some point $m\in M$, then we have 
\[
\mbox{Ric}_m^g(x,y) = \sum_{i=1}^n g_m\left(R^g(x,e_i)y,e_i\right) = \sum_{i=1}^n R_m^g\left(x,e_i,y,e_i\right)
\qquad \forall x,y \in T_mM,
\] 
and
\[
 \mbox{Scal}^g_m = \sum_{i,j=1}^n R_m^g\left(e_i,e_j,e_i,e_j\right).
\]
So, if $x=y=e_i$ then $\mbox{Ric}^g(x,x)$ corresponds to the sum of sectional curvatures of all planes with bases $(e_i,e_j)$ for $j\in \{1, \ldots, n\} \setminus \{i\}$ and if $n=2$, we have $\mbox{\rm Ric}^g = \kappa \, g$ where
$\kappa$ is the Gaussian curvature on $M$.

Assume now that we work on a neighborhood of a point $m\in M$ with local coordinates $(x_1, \ldots, x_n)$ and denote by $g_{ij}$ and $g^{ij}$ the metric $g$ and its inverse $g^{-1}$ in these coordinates.  If two vector fields $X,Y\in \Gamma(TM)$ are given in local
coordinates by
\[
X= \sum_{i=1}^n X^i \partial_i, \quad Y= \sum_{i=1}^n Y^i
\partial_i,
\]
then $\nabla^g_XY$ is given by
\begin{eqnarray*}
\nabla^g_XY= \sum_{i=1}^n \left( \sum_{j=1}^n X^j \partial_j Y^i +
\sum_{j,k=1}^n \Gamma_{jk}^i X^jY^k\right) \partial_i,
\end{eqnarray*}
where the Christoffel symbols $\Gamma_{ij}^l$, for $i,j,l$ in
$\{1,\ldots,n\}$, are defined as
\begin{eqnarray*}
\Gamma_{ij}^l = \frac{1}{2} \sum_{p=1}^n \Bigl( g^{lp} \left(
\partial_i g_{jp} + \partial_j g_{ip} - \partial_p g_{ij}
\right)\Bigr).
\end{eqnarray*}
Thus, if  $X=\partial_i, Y=\partial_j, Z=\partial_k$, then we
have
\[
\nabla_X^gZ=  \sum_{l=1}^n \left(  \Gamma_{ik}^l \right) \partial_l
\quad \mbox{and} \quad \nabla_Y^gZ=  \sum_{l=1}^n \left(
\Gamma_{jk}^l \right) \partial_l,
\]
so that
\begin{eqnarray*}
R^g(X,Y)Z & = & \nabla^g_{Y}\left(\nabla^g_XZ\right)-\nabla^g_{X}\left(\nabla^g_YZ\right)+\nabla^g_{[X,Y]}Z \\
& = &  \nabla^g_{Y}\left(\nabla^g_XZ\right)-\nabla^g_{X}\left(\nabla^g_YZ\right)\\
&= &  \sum_{p=1}^n \left( \partial_j \Gamma_{ik}^p +  \sum_{r=1}^n  \Gamma_{jr}^p \Gamma_{ik}^r\right) \partial_p -  \sum_{p=1}^n \left( \partial_i \Gamma_{jk}^p +  \sum_{r=1}^n  \Gamma_{ir}^p \Gamma_{jk}^r\right) \partial_p \\
& = & \sum_{p=1}^n R_{ijk}^p \partial_p,
\end{eqnarray*}
with
\[
R_{ijk}^p =    \partial_j \Gamma_{ik}^p - \partial_i
\Gamma_{jk}^p+  \sum_{r=1}^n \Bigl( \Gamma_{jr}^p \Gamma_{ik}^r -
\Gamma_{ir}^p \Gamma_{jk}^r \Bigr).
\]
In conclusion, if $X,Y$ are two vector fields with $X=x=\partial_i, Z=z=\partial_k$ at $m$, then we have
\begin{eqnarray}\label{RICCIformula}
\mbox{Ric}^g_{ik} :=\mbox{Ric}^g(x,z) = \sum_{j=1}^n R_{ijk}^j = \sum_{j=1}^n \left(
\partial_j \Gamma_{ik}^j - \partial_i \Gamma_{jk}^j+  \sum_{r=1}^n
\left( \Gamma_{jr}^j \Gamma_{ik}^r -   \Gamma_{ir}^j \Gamma_{jk}^r
\right)\right)
\end{eqnarray}
for all $i,k=1, \ldots,n$ 
and moreover the scalar curvature is given by 
\begin{eqnarray}\label{SCALARformula}
\mbox{Scal}^g :=  \sum_{i,j=1}^n g^{ij}\mbox{Ric}^g_{ij}.
\end{eqnarray}

\subsection{Ricci and scalar curvatures in the general case}\label{APP1GEN}

Keeping the notations of the previous section we can decompose the Ricci tensor as the sum of two terms, one which is composed of derivatives of second order of $g$ and the other gathering terms with derivatives of order one. 
 
\begin{proposition}\label{prop:coeff_de_Ricci}
Each term of the Ricci tensor can be written
for any $i,k\in \{1, \ldots,n\}$ as
\[
\mbox{\rm Ric}_{ik}^{g}=\mathcal{R}_{ik}^{g,2}+\mathcal{R}_{ik}^{g,1},
\]
with
\begin{equation*}
\mathcal{R}_{ik}^{g,2}=\frac{1}{2}\sum_{j,p=1}^{n}g^{jp}\left(\partial_{jk}g_{ip}-\partial_{ik}g_{jp}-\partial_{jp}g_{ik}+\partial_{ip}g_{jk}\right)
\end{equation*}
and
\begin{align*}\label{FormRicci1}
\mathcal{R}^{g,1}_{ik}  &=\frac{1}{4}   \sum_{j,l,p,q=1}^n g^{jl}
g^{pq} \bigg[- \partial_j g_{lq} \left(\partial_{i} g_{kp} +
\partial_{k} g_{ip}- \partial_{p} g_{ik}\right) +  \partial_i g_{lq} \left(\partial_{j} g_{kp} + \partial_{k} g_{jp} - \partial_{p} g_{jk}\right)\\
 &+   \left(\partial_{q} g_{jl} - \partial_{l} g_{jq}\right)
\left(\partial_{i} g_{kp} + \partial_{k} g_{ip} - \partial_{p}
g_{ik}\right) -    \left( \partial_{q} g_{il} - \partial_{l}
g_{iq}\right) \left(\partial_{j} g_{kp} + \partial_{k} g_{jp} -
\partial_{p} g_{jk}\right)\bigg].
\end{align*}
\end{proposition}

\begin{proof}[Proof of Proposition \ref{prop:coeff_de_Ricci}]
By (\ref{RICCIformula}), we have for any $i,k$ in $\{1, \ldots,n\}$,
\begin{align*}
\mbox{Ric}_{ik}^{g}&=\sum_{j=1}^n\left(\partial_j \Gamma_{ik}^j-\partial_i
\Gamma_{jk}^j+\sum_{r=1}^n\left(\Gamma_{jr}^j\Gamma_{ik}^r-\Gamma_{ir}^j\Gamma_{jk}^r\right)\right)\\
&=\frac{1}{2}\sum_{j,p=1}^n\Bigg[(\partial_j
g^{jp})(\partial_ig_{kp}+\partial_k g_ip-\partial_p
g_{ik})+g^{jp}(\partial_{ji}g_{kp}+\partial_{jk}g_{ip}-\partial_{jp}g_{ik})\\
& \qquad -(\partial_i g^{jp})(\partial_jg_{kp}+\partial_k g_jp-\partial_p
g_{jk})-g^{jp}(\partial_{ij}g_{kp}+\partial_{ik}g_{jp}-\partial_{ip}g_{jk})\Bigg]\\
& \qquad +\sum_{j,r=1}^n\left(\Gamma_{jr}^j\Gamma_{ik}^r-\Gamma_{ir}^j\Gamma_{jk}^r\right) = \mathcal{R}_{ik}^{g,2}+\mathcal{R}_{ik}^{g,1},
\end{align*}
where $\mathcal{R}_{ik}^{g,2}, \mathcal{R}_{ik}^{g,1}$ are defined by
\begin{equation*}
\mathcal{R}_{ik}^{g,2}=\frac{1}{2}\sum_{j,p=1}^{n}g^{jp}\left(\partial_{jk}g_{ip}-\partial_{ik}g_{jp}-\partial_{jp}g_{ik}+\partial_{ip}g_{jk}\right),
\end{equation*}
and
\begin{multline}\label{R1}
\mathcal{R}_{ik}^{g,1}=\frac{1}{2}\sum_{j,p=1}^n\Bigg[(\partial_j
g^{jp})(\partial_ig_{kp}+\partial_k g_ip-\partial_p
g_{ik})-(\partial_i g^{jp})(\partial_jg_{kp}+\partial_k
g_jp-\partial_p
g_{jk})\Bigg]\\
+\sum_{j,r=1}^n\left(\Gamma_{jr}^j\Gamma_{ik}^r-\Gamma_{ir}^j\Gamma_{jk}^r\right).
\end{multline}
Now, we want to express the derivatives of coefficients of $g^{-1}$
in term of coefficients of $g$ and their derivatives. Writing that
$g^{-1}$ is the inverse of $g$, we obtain
\[
\delta_{ij} = \sum_{k=1}^n g^{ik}g_{kj}
\]
for any $i,j=1, \ldots,n$. Taking a partial derivative $\partial_l$
in the above expression yields for any $i,j, l=1, \ldots,n$
\[
\sum_{k=1}^n  \Bigl( g_{kj} \partial_l g^{ik} + g^{ik} \partial_l
g_{kj}\Bigr) =0,
\]
then multiplying with $g^{jm}$, we have
\[
\sum_{k=1}^n \Bigl(  g^{jm} g_{kj} \partial_l g^{ik} + g^{jm}g^{ik}\partial_l g_{kj}\Bigr)=0
\]
which gives
\[
-  \sum_{j,k=1}^n  \Bigl(g^{jm}g^{ik} \partial_l g_{kj}
\Bigr)=\sum_{j,k=1}^n \Bigl(g^{jm} g_{kj} \partial_l
g^{ik}\Bigr)=\sum_{k=1}^n \Bigl(\delta_{mk} \partial_l g^{ik}\Bigr)
=
\partial_l g^{im}.
\]
Coming back to \eqref{R1}, we obtain
\begin{align*}
\mathcal{R}_{ik}^{g,1}=&\frac{1}{2}\sum_{j,l,p,q}g^{jl}g^{pq}\bigg[-\partial_j
g_{lq}\Big(\partial_ig_{kp}+\partial_k g_{ip}-\partial_p
g_{ik}\Big)+\partial_i g_{lq}\Big(\partial_jg_{kp}+\partial_k
g_{jp}-\partial_p g_{jk}\Big)\bigg]\\
& \qquad + \frac{1}{4}\sum_{j,l,p,q}g^{jl}g^{pq}\bigg[\big(\partial_j
g_{pl}+\partial_pg_{jl}-\partial_l g_{jp}\big)\big(\partial_i
g_{kq}+\partial_kg_{iq}-\partial_q g_{ik}\big)\\
& \qquad \qquad -\big(\partial_i g_{pl}+\partial_pg_{il}-\partial_l
g_{ip}\big)\big(\partial_j g_{kq}+\partial_kg_{jq}-\partial_q
g_{jk}\big)\bigg].
\end{align*}
The required formula follows by interchanging $p$ and $q$ in the last sum and simplifying with the first sum.
\end{proof}

Applying the above formulas to (\ref{SCALARformula}) gives

\begin{proposition}\label{prop:coeff_scalaire}
The Scalar curvature of a metric $g$ can be written as,
\[
\mbox{\rm Scal}^g:= \mathcal{S}^{g,2} + \mathcal{S}^{g,1},
\]
with
\begin{equation*}
\mathcal{S}^{g,2}=\sum_{i,k,j,p=1}^{n}\left(g^{ik}g^{jp}-g^{ip}g^{jk}\right)\partial_{jk}g_{ip}
\end{equation*}
and
\begin{multline*}
\mathcal{S}^{g,1}\\
=\frac{1}{4}\sum_{i,k,j,l,p,q=1}^ng^{ik}g^{jl}
g^{pq}\Bigg(-4\partial_{j} g_{lq} \partial_{i} g_{kp}+4\partial_{j}
g_{lq} \partial_{p} g_{ik}+3\partial_{i} g_{lq} \partial_{k}
g_{jp}-\partial_{q} g_{jl}
\partial_{p} g_{ik}-2\partial_{q}
g_{il} \partial_{j} g_{kp}\Bigg).
\end{multline*}
\end{proposition}

\begin{proof}[Proof of Proposition \ref{prop:coeff_scalaire}]
One has, $$Scal^g:=\sum_{i,k}g^{ik}Ric_{ik}^g.$$ We can write
$$Scal^g:=\mathcal{S}^{g,1}+\mathcal{S}^{g,2},$$
where
\begin{equation}\label{Formscal2}
\mathcal{S}^{g,2}=\sum_{i,k}g^{ik}\mathcal{R}_{ik}^{g,2}=
\frac{1}{2}\sum_{i,k,j,p=1}^{n}g^{ik}g^{jp}\left(\partial_{jk}g_{ip}-\partial_{ik}g_{jp}-\partial_{jp}g_{ik}+\partial_{ip}g_{jk}\right),
\end{equation}
and
\begin{align}\label{Formscal1}
\mathcal{S}^{g,1}
&=\sum_{i,k}g^{ik}\mathcal{R}_{ik}^{g,1}\nonumber\\
&=\frac{1}{4} \sum_{i,k,j,l,p,q=1}^n g^{ik}g^{jl} g^{pq} \bigg[-
\partial_j g_{lq} \left(\partial_{i} g_{kp} +
\partial_{k} g_{ip}- \partial_{p} g_{ik}\right) +  \partial_i g_{lq} \left(\partial_{j} g_{kp} + \partial_{k} g_{jp} - \partial_{p} g_{jk}\right)\nonumber\\
 &+   \left(\partial_{q} g_{jl} - \partial_{l} g_{jq}\right)
\left(\partial_{i} g_{kp} + \partial_{k} g_{ip} - \partial_{p}
g_{ik}\right) -    \left( \partial_{q} g_{il} - \partial_{l}
g_{iq}\right) \left(\partial_{j} g_{kp} + \partial_{k} g_{jp} -
\partial_{p} g_{jk}\right)\bigg].
\end{align}
To simplify the relation \eqref{Formscal1}, let us set
$$\mathcal{S}^{g,1}:=\frac{1}{4}(A+B+C+D),$$
where
\begin{eqnarray*}
A &:=& \sum_{i,k,j,l,p,q=1}^n
g^{ik}g^{jl} g^{pq} \Big(- \partial_j g_{lq} \left(\partial_{i}
g_{kp} + \partial_{k} g_{ip}- \partial_{p} g_{ik}\right)\Big)\\
B &:=& \sum_{i,k,j,l,p,q=1}^n
g^{ik}g^{jl} g^{pq}\Big(\partial_i g_{lq} \left(\partial_{j} g_{kp}
+ \partial_{k} g_{jp} - \partial_{p} g_{jk}\right)\Big)\\
C &:=& \sum_{i,k,j,l,p,q=1}^ng^{ik}g^{jl} g^{pq}\Big(\left(\partial_{q} g_{jl} - \partial_{l} g_{jq}\right)
\left(\partial_{i} g_{kp} + \partial_{k} g_{ip} - \partial_{p}
g_{ik}\right)  \Big)\\
D &:=& \sum_{i,k,j,l,p,q=1}^ng^{ik}g^{jl} g^{pq} \Big(-\left( \partial_{q} g_{il} - \partial_{l}
g_{iq}\right) \left(\partial_{j} g_{kp} + \partial_{k} g_{jp} -
\partial_{p} g_{jk}\right)\Big).
\end{eqnarray*}
As $g$ and $g^{-1}$ are symmetric, we have
\[
\sum_{i,k=1}^n g^{ik} \partial_ig_{kp} = \sum_{i,k=1} g^{ik}
\partial_kg_{ip}\]
then we get
$$A=\sum_{i,k,j,l,p,q=1}^n
g^{ik}g^{jl} g^{pq} \Big(- \partial_j g_{lq} \left(2\partial_{i}
g_{kp}- \partial_{p} g_{ik}\right)\Big).$$ Using the same idea to
exchange $(jl)$ and $(pq)$, we obtain
\[
\sum_{j,l,p,q=1}^n g^{jl} g^{pq} \partial_i g_{lq}\partial_jg_{kp} =
\sum_{j,l,p,q=1}^n g^{jl} g^{pq} \partial_i g_{lq}\partial_pg_{jk}\]
one has
$$B=\sum_{i,k,j,l,p,q=1}^n
g^{ik}g^{jl} g^{pq}\Big(\partial_i g_{lq}
\partial_{k} g_{jp} \Big).$$
Now exchanging $(ik)$ to $(jl)$ and $p$ to $q$, we have
\[\sum_{i,j,k,l,p,q}g^{ik}g^{jl} g^{pq} \partial_j g_{lq}\partial_pg_{ik} =
\sum_{i,j,k,l,p,q}g^{ik}g^{jl} g^{pq} \partial_i
g_{kp}\partial_qg_{jl},\quad \sum_{i,k=1}^n g^{ik} \partial_ig_{kp}
= \sum_{i,k=1} g^{ik}
\partial_kg_{ip}\]
we have
$$C= \sum_{i,k,j,l,p,q=1}^ng^{ik}g^{jl} g^{pq}\Big(2\partial_{j}g_{lq}\partial_{p} g_{ik} - \partial_{q} g_{jl}\partial_{p} g_{ik} -\partial_l g_{jq}\left( 2\partial_{i} g_{kp} - \partial_{p}
g_{ik}\right)  \Big).$$ Finally, by

\[ \sum_{i,k j,l=1}^n g^{ik} g^{jl}  \partial_q
g_{il}\partial_jg_{kp} = \sum_{j,l,p,q=1}^n g^{ik} g^{jl} \partial_q
g_{il}\partial_kg_{jp},
\]
\[
\sum_{i,k,p,q=1}^n g^{ik} g^{pq} \partial_l g_{iq}\partial_kg_{jp} =
\sum_{i,k,p,q=1}^n g^{ik} g^{pq} \partial_l g_{iq}\partial_pg_{jk},
\]
\[
\sum_{i,k,p,q=1}^n g^{ik} g^{pq} \partial_q
g_{il}\partial_pg_{jk}=\sum_{i,k,p,q=1}^n g^{ik} g^{pq} \partial_i
g_{lq}\partial_kg_{jp}
\]
and
\[\sum_{i,k,j,l=1}^n g^{ik} g^{jl} \partial_l
g_{iq}\partial_jg_{kp}= \sum_{i,k,j,l=1}^n g^{ik} g^{jl} \partial_i
g_{lq}\partial_kg_{jp} ,
\]
we have
$$D=   2\sum_{i,k,j,l,p,q=1}^ng^{ik}g^{jl} g^{pq} \Big(-\partial_{q} g_{il}\partial_{j} g_{kp} + \partial_{i}
g_{lq} \partial_{k} g_{jp}\Big).$$ Substituting the formula $A,B,C$
and $D$ in $\mathcal{S}^{g,1}$, we deduce by the use of the
following relation
 \[\sum_{j,l=1}  g^{jl}  \partial_j g_{lq} = \sum_{j,l=1}
g^{jl}  \partial_l g_{jq},
\] that
\begin{multline*}
\mathcal{S}^{g,1}=\\
\frac{1}{4}\sum_{i,k,j,l,p,q=1}^ng^{ik}g^{jl} g^{pq}\Bigg(-4\partial_{j}
g_{lq} \partial_{i} g_{kp}+4\partial_{j} g_{lq} \partial_{p}
g_{ik}+3\partial_{i} g_{lq} \partial_{k} g_{jp}-\partial_{q} g_{jl}
\partial_{p} g_{ik}-2\partial_{q}
g_{il} \partial_{j} g_{kp}\Bigg). 
\end{multline*}
Returning to the formula of $\mathcal{S}^{g,2}$, we use the following identifications
\[\sum_{i,k,j,p} g^{ik}g^{jp}\partial_{ik}g_{jp}=\sum_{i,k,j,p} g^{ik}g^{jp}\partial_{jp}g_{ik},\quad
\sum_{i,k,j,p} g^{ik}g^{jp}\partial_{jk}g_{ip}=\sum_{i,k,j,p}
g^{ik}g^{ip}\partial_{ip}g_{jk}.\] So, relation \eqref{Formscal2}
can be written as
$$\mathcal{S}^{g,2}=\sum_{i,k,j,p=1}^{n}g^{ik}g^{jp}\left(\partial_{jk}g_{ip}-\partial_{jp}g_{ik}\right),$$
from which we obtain
$$\mathcal{S}^{g,2}=\sum_{i,k,j,p=1}^{n}\left(g^{ik}g^{jp}-g^{ip}g^{jk}\right)\partial_{jk}g_{ip}.$$
\end{proof}

\subsection{Ricci and scalar curvature of diagonal metrics}\label{APPgdiag}

The purpose of this section is to specify the formulas of Ricci and scalar curvatures in the case of a diagonal metric. So we assume that there are smooth functions $f_1, \ldots, f_n: M \rightarrow \R$ such that 
\begin{eqnarray}\label{gdiag}
\begin{array}{l}
g_{ij}\equiv 0 \quad\forall i, j \in \{1, \ldots,n\} \mbox{ with } i \neq j,\\
g_{ii} = e^{2f_i} \quad \forall i=1, \ldots, n.
\end{array}
\end{eqnarray}
 By noting that
\[
\partial_k g_{ii} = 2 e^{2f_i} \partial_k f_i \quad \mbox{and} \quad \partial_{kl} g_{ii} = 2 e^{2f_i} \partial_{kl} f_i + 4 e^{2f_i} \partial_{k} f_i \partial_{l} f_i,
\]
for all $i,k,l$, Proposition \ref{prop:coeff_de_Ricci} yields:

\begin{proposition}
Assume that $g$ is diagonal of the form (\ref{gdiag}) then we have for any $i,k\in \{1, \ldots,n\}$
\[
\mbox{\rm Ric}_{ik}^{g}=\tilde{\mathcal{R}}_{ik}^{g,2}+\tilde{\mathcal{R}}_{ik}^{g,1},
\]
where 
\begin{eqnarray*}
\tilde{\mathcal{R}}_{ii}^{g,2} = - \sum_{j=1, j \neq i}^{n}\left(\partial_{ii}f_{j}+e^{2(f_i-f_j)}\partial_{jj}f_i\right)
\end{eqnarray*}
for all $i= 1, \ldots,n$,
\begin{eqnarray*}
\tilde{\mathcal{R}}_{ik}^{g,2} = - \sum_{j=1, j \notin \{i,k\}}^{n} \partial_{ik}f_{j} 
\end{eqnarray*}
for all $i,k\in \{1, \ldots,n\}$ with $i\neq k$,
\begin{multline*}
\tilde{\mathcal{R}}_{ii}^{g,1}=   \sum_{j=1, j\neq i}^n e^{2(f_i-f_j)} \partial_j f_{j} \partial_{j} f_{i} - \sum_{j=1, j\neq i}^n  (\partial_i f_{j})^2 +   \sum_{j=1, j\neq i}^n \partial_{i} f_{j}  \partial_{i} f_{i}\\
-   \sum_{j,p=1, j\neq i, p\notin \{i,j\}}^n e^{2(f_i-f_p)} \partial_{p} f_{j}  \partial_{p}f_{i} -  \sum_{j=1, j\neq i}^n e^{2(f_i-f_j)}  (\partial_{j} f_{i})^2
\end{multline*}
for all $i= 1, \ldots,n$, and 
\[
\tilde{\mathcal{R}}^{g,1}_{ik} = - \sum_{j=1, j\notin \{i,k\}}^n   \partial_i f_{j} \partial_{k} f_{j} +   \sum_{j=1, j\notin \{i,k\}}^n \partial_{k} f_{j} \partial_{i} f_{k} + \sum_{j=1, j\notin \{i,k\}}^n \partial_{i} f_{j} \partial_{k} f_{i} 
\]
for all $i,k \in \{1, \ldots,n\}$ with $i\neq k$. 
\end{proposition}

Morever, Proposition \ref{prop:coeff_scalaire} gives:

\begin{proposition}
Assume that $g$ is diagonal of the form (\ref{gdiag}) then we have
\[
\mbox{\rm Scal}^{g}=\tilde{\mathcal{S}}^{g,2}+\tilde{\mathcal{S}}^{g,1},
\]
with
\begin{eqnarray*}
\tilde{\mathcal{S}}^{g,2} =  - 2 \sum_{i,j=1, j\neq i}^n e^{-2f_i}\partial_{ii}f_{j}
\end{eqnarray*}
and
\begin{eqnarray*}
\tilde{\mathcal{S}}^{g,1} =  2 \sum_{i, j=1, i\neq j}^n e^{-2f_i} \partial_i f_{i} \partial_{i} f_{j} - 2 \sum_{i, j=1, j\neq i}^n e^{-2f_i} (\partial_i f_{j})^2 
-   \sum_{i, j,p=1, j\neq i, p\notin \{i,j\}}^n e^{-2f_p} \partial_{p} f_{j}  \partial_{p}f_{i}.
\end{eqnarray*}
\end{proposition}

\subsection{Scalar curvature of a metric with diagonal perturbation}\label{APPPertMetric}
Given a Riemannian metric $\bar{g}=(\bar{g}_{ij})$ in $\R^n$ with $n\geq 3$ of the form 
\begin{eqnarray}\label{gorthonormal}
\bar{g} = \bar{A}^\mathsf{tr}\bar{A} \quad \mbox{with} \quad \bar{A}_{\bar{x}} = \mbox{Id}_n,
\end{eqnarray}
where $\bar{x}\in \R^n$ is fixed and $x\in \R^n\mapsto \bar{A}_x\in GL_n(\R)$ is a smooth mapping, we consider a perturbation $\hat{g}$ of $\bar{g}$ on a neighborhood of $\bar{x}$ of the form 
\begin{eqnarray}\label{bargperturb}
\hat{g} = \bar{A}^\mathsf{tr}D\bar{A}
\end{eqnarray}
where $D$ is a $n\times n$ diagonal matrix with coefficients $(e^{2f_1}, \ldots,e^{2f_n})$ and $f_1, \ldots, f_n$ are smooth functions on a neighborhood of $\bar{x}$ valued in a neighborhood of $0\in \R$ (in such a way that $\hat{g}$ remains positive definite near $\bar{x}$). Using that (we denote the coefficientss of $\bar{A}$ by $(\bar{a}_{ij})$)
\[
\hat{g}_{ij} = \bar{g}_{ij} + \sum_{r=1}^n \left(e^{2f_r}-1\right) \bar{a}_{ri}\bar{a}_{rj}  \qquad \forall i=j=1, \ldots,n,
\]
we verify that the first and second-order partial derivatives of $\hat{g}$ are given by 
\begin{multline}\label{7Mai0}
\partial_k \hat{g}_{ij} =  \partial_k \bar{g}_{ij} + 2 \sum_{r=1}^n \Bigl(e^{2f_r} \left(\partial_{k} f_r\right) \bar{a}_{ri}\bar{a}_{rj}\Bigr) + \sum_{r=1}^n \Bigl(\left(e^{2f_r}-1\right) \left[ \left(\partial_k \bar{a}_{ri}\right) \bar{a}_{rj} + \bar{a}_{ri} \left( \partial_k \bar{a}_{rj}\right) \right] \Bigr)
\end{multline}
and
\begin{multline}\label{7Mai1}
\partial_{kl} \hat{g}_{ij} = \partial_{kl} \bar{g}_{ij} + 2 \sum_{r=1}^n \Bigl(e^{2f_r} \left(\partial_{kl} f_r\right) \bar{a}_{ri}\bar{a}_{rj}\Bigr) + 4 \sum_{r=1}^n \Bigl(e^{2f_r} \left(\partial_{k} f_r\right)  \left(\partial_{l} f_r\right) \bar{a}_{ri} \bar{a}_{rj}\Bigr) \\
+ 2 \sum_{r=1}^n \Bigl(e^{2f_r} \left(\partial_{k} f_r\right)  \left[ \left(\partial_l \bar{a}_{ri}\right) \bar{a}_{rj} + \bar{a}_{ri} \left( \partial_l \bar{a}_{rj}\right) \right]\Bigr) + 2 \sum_{r=1}^n \Bigl(e^{2f_r} \left(\partial_{l} f_r\right)  \left[ \left(\partial_k \bar{a}_{ri}\right) \bar{a}_{rj} + \bar{a}_{ri} \left( \partial_k \bar{a}_{rj}\right) \right]\Bigr)\\
+ \sum_{r=1}^n \Bigl(\left(e^{2f_r}-1\right) \left[ \left(\partial_{kl} \bar{a}_{ri}\right) \bar{a}_{rj} + \left(\partial_k \bar{a}_{ri}\right) \left( \partial_l \bar{a}_{rj}\right) + \left( \partial_l \bar{a}_{ri}\right) \left( \partial_k \bar{a}_{rj}\right)+ \bar{a}_{ri} \left( \partial_{kl} \bar{a}_{rj}\right) \right] \Bigr),
\end{multline}
for all $i,j,k,l=1, \ldots,n$. Then, by applying the scalar curvature formulas given in Proposition \ref{prop:coeff_scalaire} at $\bar{x}$ where $\bar{A}_{\bar{x}} = \mbox{Id}_n$ and $\hat{g}_{\bar{x}}=D(\bar{x})$, we obtain:

\begin{proposition}\label{PROPgorthonormal}
Assume that $g$ is diagonal of the form (\ref{gorthonormal}) then we have 
\begin{eqnarray}\label{Scalhatg}
\mbox{Scal}^{\hat{g}}_{\bar{x}} = \mbox{Scal}^{\bar{g}}_{\bar{x}} + B^{\hat{g}}_{\bar{x}} + Q^{\hat{g}}_{\bar{x}} +  E^{\hat{g}}_{\bar{x}},
\end{eqnarray}
where $B^{\hat{g}}_{\bar{x}}$, which is linear in the second-order derivatives of $f_1, \ldots, f_n$ at $\bar{x}$ is given by
\begin{eqnarray}\label{ScalhatgB}
B^{\hat{g}}_{\bar{x}} = - 2\sum_{i,j=1, i\neq j}^n \Bigl(  e^{-2f_i(\bar{x})} \partial_{i}^2 f_j(\bar{x}) \Bigr)
\end{eqnarray}
where $Q^{\hat{g}}_{\bar{x}}$, which is quadratic in the first-order derivatives of $f_1, \ldots, f_n$ at $\bar{x}$, is given by
\begin{multline}\label{ScalhatgQ}
Q^{\hat{g}}_{\bar{x}} = - 2\sum_{i,j=1, i\neq j}^n \Bigl(  e^{-2f_i(\bar{x})} \left(\partial_{i} f_j(\bar{x}) \right)^2 \Bigr)\\
-  \sum_{i,j,k=1, i\neq j}^n \Bigl( e^{-2f_k(\bar{x})}  \partial_k f_{i}(\bar{x})  \partial_{k} f_{j}(\bar{x}) \Bigr)  + 4 \sum_{i,j=1,i\neq j}^n  \Bigl( e^{-2f_i(\bar{x})} \partial_{i} f_{i}(\bar{x}) \partial_i f_{j} (\bar{x}) \Bigr),
\end{multline}
and where $E^{\hat{g}}$ is an error term which can be written as 
\begin{multline}\label{ScalhatgE}
E^{\hat{g}}_{\bar{x}} = \Phi_0 \left(f_1(\bar{x}), \cdots, f_n(\bar{x}), \partial^1 \bar{A}(\bar{x}), \partial^2 \bar{A}(\bar{x})\right) \\
+ \sum_{r=1}^n \sum_{k=1}^n \Bigl( \Phi_{i,k} \left(f_1(\bar{x}),\cdots, f_n(\bar{x}), \partial^1 \bar{A}(\bar{x}), \partial^2 \bar{A}(\bar{x})\right) \left(\partial_kf_r(\bar{x})\right)\Bigr)
\end{multline}
where $ \Phi_{0}$ and $\Phi_{1,k}, \Phi_{2,k}$ with $k=1, \ldots,n$, are smooth function in the variables $f_2, f_3, \bar{a}_{ij}$, $\partial^1_k \bar{a}_{ij}$ and $\partial^2_{kl} \bar{a}_{ij}$ such that  $\Phi_0 (0,0,\cdot, \cdot, \cdot)=0$. 
\end{proposition}

\begin{proof}[Proof of Proposition \ref{PROPgorthonormal}]
Formulas of Proposition \ref{prop:coeff_scalaire}, along with $\hat{g}_{\bar{x}}=D(\bar{x})$, yield 
\[
\mbox{\rm Scal}_{\bar{x}}^{\hat{g}}:= \mathcal{S}_{\bar{x}}^{\hat{g},2} + \mathcal{S}_{\bar{x}}^{\hat{g},1},
\]
with
\begin{eqnarray*}
\mathcal{S}^{\hat{g},2}_{\bar{x}} & = & \sum_{i,k,j,p=1}^{n}\left(\hat{g}^{ik}_{\bar{x}}\hat{g}_{\bar{x}}^{jp}- \hat{g}_{\bar{x}}^{ip} \hat{g}_{\bar{x}}^{jk}\right)\partial_{jk} \hat{g}_{ip} (\bar{x})\\
& = & \sum_{i,j=1}^{n} \hat{g}^{i}_{\bar{x}}\hat{g}_{\bar{x}}^{j} \partial_{ij} \hat{g}_{ij} (\bar{x}) - \sum_{i,j=1}^{n} \hat{g}_{\bar{x}}^{i} \hat{g}_{\bar{x}}^{j}\partial_{jj} \hat{g}_{ii} (\bar{x}) \\
& = & \sum_{i,j=1, i\neq j}^{n} e^{-2f_i(\bar{x})} e^{-2f_j(\bar{x})} \Bigl(  \partial_{ij} \hat{g}_{ij}(\bar{x})  - \partial_{jj} \hat{g}_{ii}(\bar{x})  \Bigr)
\end{eqnarray*}
and (from now on we omit $\bar{x}$)
\begin{multline*}
\mathcal{S}_{\bar{x}}^{\hat{g},1} =\frac{1}{4}\sum_{i,j,p=1}^n
e^{-2f_i}e^{-2f_j}e^{-2f_p}
\Bigg(-4\partial_{j} \hat{g}_{jp} \partial_{i} \hat{g}_{ip}
+4\partial_{j} \hat{g}_{jp} \partial_{p} \hat{g}_{ii}
+3\partial_{i} \hat{g}_{jp} \partial_{i} \hat{g}_{jp}\\
-\partial_{p} \hat{g}_{jj}\partial_{p} \hat{g}_{ii}
-2\partial_{p} \hat{g}_{ij} \partial_{j} \hat{g}_{ip}\Bigg).
\end{multline*}
The first term $\mathcal{S}^{\hat{g},2}_{\bar{x}}$ can be written as 
\begin{multline*}
\mathcal{S}^{\hat{g},2}_{\bar{x}}  =  \sum_{i,j=1, i\neq j} \Bigl(  \partial_{ij} \bar{g}_{ij}  - \partial_{jj} \bar{g}_{ii}  \Bigr) \\
 + \sum_{i,j=1, i\neq j}^{n} e^{-2f_i} e^{-2f_j} \Bigl(  \partial_{ij} \hat{g}_{ij}  - \partial_{jj} \hat{g}_{ii}  \Bigr) - \sum_{i,j=1, i\neq j} \Bigl(  \partial_{ij} \bar{g}_{ij}  - \partial_{jj} \bar{g}_{ii}  \Bigr),
\end{multline*}
which by using  (\ref{7Mai1}) is equal to
\begin{eqnarray*}
\mathcal{S}^{\hat{g},2}_{\bar{x}}  =  \sum_{i,j=1, i\neq j} \Bigg( \Bigl(  \partial_{ij} \bar{g}_{ij}  - \partial_{jj} \bar{g}_{ii}  \Bigr) - 2   \Bigl(  e^{-2f_j} \partial_{j}^2 f_i \Bigr) - 4 \Bigl(e^{-2f_j} \left(\partial_{j} f_i\right)  \left(\partial_{j} f_i\right)\Bigr) \Bigg) + \Delta_2
\end{eqnarray*}
with

\begin{multline*}
\Delta_2 = \sum_{i,j=1, i\neq j}^{n} \left( e^{-2f_i} e^{-2f_j}-1\right) \Bigl(  \partial_{ij} \bar{g}_{ij}  - \partial_{jj} \bar{g}_{ii}  \Bigr) \\
+  \sum_{i,j=1, i\neq j} \sum_{r=i,j} e^{-2f_i} e^{-2f_j} \Bigg( 2  e^{2f_r} \partial_{i} f_r\left[ \partial_j \bar{a}_{ri} \, \delta_{rj} + \delta_{ri} \, \partial_j \bar{a}_{rj} \right] + 2  e^{2f_r} \partial_{j} f_r  \left[ \partial_i \bar{a}_{ri}\, \delta_{rj} + \delta_{ri} \, \partial_i \bar{a}_{rj} \right]\\
+ \left(e^{2f_r}-1\right) \left[ \partial_{ij} \bar{a}_{ri}\, \delta_{rj} + \partial_i \bar{a}_{ri}\,  \partial_j \bar{a}_{rj} + \partial_j \bar{a}_{ri}\, \partial_i \bar{a}_{rj}+ \delta_{ri} \, \partial_{ij} \bar{a}_{rj} \right] \Bigr)  \Bigg)\\
-  \sum_{i,j=1, i\neq j} \Bigg( 2 e^{2f_i} \partial_{j} f_i  \left[ \partial_j \bar{a}_{ii}  +   \partial_j \bar{a}_{ii} \right] + 2 e^{2f_i} \partial_{j} f_i  \left[ \partial_j \bar{a}_{ii} +   \partial_j \bar{a}_{ii} \right]\\
+ \left(e^{2f_i}-1\right) \left[ \partial_{j}^2 \bar{a}_{ii} + \partial_j \bar{a}_{ii}\,  \partial_j \bar{a}_{ii} +  \partial_j \bar{a}_{ii}\,  \partial_j \bar{a}_{ii}+   \partial_{jj}^2 \bar{a}_{ii} \right] \Bigg),
\end{multline*}
where for all $i,j,k,l,$ $\partial_{kl} \bar{g}_{ij}$ can be written as
\[
\partial_{kl} \bar{g}_{ij} =  \sum_{r=1}^n \Bigl( \partial_{kl} \bar{a}_{ri}\, \delta_{rj} + \partial_k \bar{a}_{ri}\, \partial_l \bar{a}_{rj} + \partial_l \bar{a}_{ri}\, \partial_k \bar{a}_{rj}+ \delta_{ri} \, \partial_{kl} \bar{a}_{rj}  \Bigr).
\]

Hence by pluging the above formula for $\partial_{kl} \hat{g}_{ij}$, we obtain
\begin{equation*}
\mathcal{S}^{g,2}=\sum_{i,k,j,p=1}^{n}\left(\hat{g}^{ik}\hat{g}^{jp}-\hat{g}^{ip}\hat{g}^{jk}\right)\partial_{jk}g_{ip}
\end{equation*}
\end{proof}

\section{Proof of Lemma \ref{PROPTnGenLEM}}\label{PROPTnGenLEMProof}
Let $x\in \R^n$ be fixed. The pullback $\bar{g}$ of $\tilde{g}_0$ by the affine diffeomorphism $\Psi: \R^n \rightarrow \R^n$ defined by 
\[
\Psi(z) := x + P_{x}z \qquad \forall z \in \R^n,
\]
on a neighborhood of $0=\Psi^{-1}(x)$ is given by (using (\ref{5fev1}))
\[
\bar{g}_z = P_{x}^\mathsf{tr} \left(\tilde{g}_0\right)_{\Psi^{x}(z)} P_{x} = P_{x}^\mathsf{tr}\left(Q_{\Psi(z)}^\mathsf{tr}Q_{\Psi(z)}\right) P_{x} = \left(Q_{\Psi(z)}P_{x}\right)^\mathsf{tr} \left(Q_{\Psi(z)}P_{x}\right) \mbox{ where } P_{x}Q_{\Psi(0)}=\mbox{Id}_n,
\]
and the pullback $\hat{g}$ of $\tilde{g}$ by $\Psi$ by (using (\ref{gtildeformula}))
\begin{eqnarray*}
\hat{g}_z = P_{x}^\mathsf{tr}\tilde{g}_{\Psi(z)} P_{x} & = & P_{x}^\mathsf{tr} Q_{\Psi(z)}^\mathsf{tr} D \left(\Psi(z)\right) Q_{\Psi(z)}P_{x}\\
& = & \left(Q_{\Psi(z)}P_{x}\right)^\mathsf{tr} D \left(x+ P_{x} z\right) \left(Q_{\Psi(z)}P_{x}\right).
\end{eqnarray*}
So, by setting $A(z):=P_{x}Q_{\Psi(z)}$ which verifies $A(0)=\mbox{Id}_n $ we can apply the formula of scalar curvature provided in Appendix \ref{APPPertMetric}. By noting that $D (x+  P_{x}z) $ is a diagonal matrix with diagonal coefficients $(1, e^{2f_2}, e^{2f_3},1, \ldots,1)$ where
\begin{eqnarray}\label{formfTn}
f_i(z) = h_i  \left(x+ P_{x}z\right) \qquad \forall i=2,3
\end{eqnarray}
for $z$ on a neighborhood of the origin $0$, (\ref{Scalhatg})-(\ref{ScalhatgE})  yield (we set $f_k=0$ for $k\neq 2,3$)
\begin{eqnarray}\label{ScalIdproof}
\mbox{Scal}^{\tilde{g}}_{\bar{x}}  = \mbox{Scal}^{\hat{g}}_{0} = \mbox{Scal}^{\bar{g}}_{0} + B^{\hat{g}}_{0} +  Q^{\hat{g}}_{0} +  E^{\hat{g}}_{0} = \mbox{Scal}^{\tilde{g}_0}_{x} + B^{\hat{g}}_{0} +  Q^{\hat{g}}_{0} +  E^{\hat{g}}_{0},
\end{eqnarray}
with
\begin{eqnarray}\label{Bghat}
B^{\hat{g}}_{0} = - 2\sum_{k=1, k\neq 2}^n \Bigl(  e^{-2f_k(0)} \partial_{k}^2 f_2(0)\Bigr) - 2\sum_{k=1, k\neq 3}^n \Bigl(  e^{-2f_k(0)} \partial_{k}^2 f_3(0)\Bigr),
\end{eqnarray}
\begin{multline}\label{ScalhatgQ}
Q^{\hat{g}}_{0} = - 2\sum_{k=1, k\neq 2}^n  \Bigl( e^{-2f_k(0)} \left(\partial_{k} f_2(0) \right)^2\Bigr) - 2\sum_{k=1, k\neq 3}^n  \Bigl( e^{-2f_k(0)} \left(\partial_{k} f_3(0) \right)^2 \Bigr)\\
- 2 \sum_{k=1, k\notin\{2,3\}}^n \Bigl(  \partial_k f_{2}(0)  \partial_{k} f_{3}(0)\Bigr) \\
 + 2  e^{-2f_2(0)} \partial_{2} f_{2}(0) \partial_2 f_{3} (0) + 2 e^{-2f_3(0)} \partial_{3} f_{3}(0) \partial_3 f_{2} (0),
\end{multline}
 and $E^{\hat{g}}$ is an error term that can be written as 
\begin{multline*}
E^{\hat{g}}_0 = \Phi_0 \left(f_2(0),f_3(0), \partial^1 A(0), \partial^2 A(0)\right) \\
+ \sum_{k=1}^n \Bigl(\Phi_{2,k} \left(f_2(0),f_3(0), \partial^1 A(0), \partial^2 A(0)\right) \left(\partial_kf_2(0)\right)\Bigr) \\
+ \sum_{k=1}^n \Bigl(\Phi_{3,k} \left(f_2(0),f_3(0), \partial^1 A(0), \partial^2 A(0)\right) \left(\partial_kf_3(0)\right)\Bigr),
\end{multline*}
where $ \Phi_{0}$ and $\Phi_{2,k}, \Phi_{3,k}$, with $k=1, \ldots,n$, are smooth functions in the variables $f_2(0), f_3(0)$, $\partial^1_k a_{ij}(0)$ and $\partial^2_{kl} a_{ij}(0)$ such that  $\Phi_0 (0,0,\cdot, \cdot)=0$. By (\ref{formfTn}), we have
\begin{eqnarray*}
\partial_k f_i(0) = d_xh_i \Big( X^k(x) \Big) \quad \mbox{and} \quad \partial_k^2 f_i(0) = d^2_xh_i \Big( X^k(x), X^k(x) \Big) \qquad \forall k=1, \ldots,n, \, \forall i=2,3. 
\end{eqnarray*}
Hence, by (\ref{5fev2}), for every $i=2,3$, we have
\[
\partial_1 f_i(0)  =  \alpha_x \partial_1 h_i(x) + \sum_{j=2}^n  \left(P_x\right)_{j,1} \partial_j h_i (x),
\]
\[
\partial_1^2 f_i(0) = \alpha_x^2 \partial^2_1 h_i (x) + 2 \alpha_x \sum_{j=2}^n  \left(P_x\right)_{j,1} \partial^2_{1j} h_i(x) +  \sum_{j,k=2}^n  \left(P_x\right)_{j,1} \left(P_x\right)_{k,1}  \partial^2_{jk} h_i(x),
\]
and  for every $i=2,3$ and $k=2, \ldots,n$,
\[
\partial_k f_i(0) = \sum_{j=2}^n  \left(P_x\right)_{j,k} \partial_j h_i(x), \quad 
\partial_k^2 f_i(0) =  \sum_{j,l=2}^n  \left(P_x\right)_{j,k} \left(P_x\right)_{l,k}  \partial^2_{jl} h_i(x).
\]
As a result, $B^{\hat{g}}_{0}$ in (\ref{Bghat}) can be expressed as 
\begin{multline*}
B^{\hat{g}}_0 = -2 \alpha_x^2 \Bigl( \partial_1^2 h_2(x) + \partial_1^2 h_3(x) \Bigr) -4 \alpha_x  \sum_{j=2}^n  \left(P_x\right)_{j,1} \Bigl( \partial^2_{1j} h_2(x) +  \partial^2_{1j} h_3(x)\Bigr) \\
- 2 \sum_{j,k=2}^n  \left(P_x\right)_{j,1}  \left(P_x\right)_{k,1} \Bigl( \partial^2_{jk} h_2(x) + \partial^2_{jk} h_3(x)\Bigr)\\
-2 \sum_{k=3}^n \sum_{j,l=2}^n \Bigl(  e^{-2f_k(0)}  \left(P_x\right)_{j,k} \left(P_x\right)_{l,k}  \partial^2_{jl} h_2(x) \Bigr) - 2\sum_{k=2, k\neq 3}^n \sum_{j,l=2}^n\Bigl(  e^{-2f_k(0)}  \left(P_x\right)_{j,k} \left(P_x\right)_{l,k}  \partial^2_{jl} h_3(x)\Bigr)
\end{multline*}
with the two last lines depend only on $\check{J}^2_h$ and $x$.  The term $Q^{\hat{g}}_{0}$ in (\ref{ScalhatgQ}) can be written as
\begin{multline*}
Q^{\hat{g}}_{0} = - 2 \alpha_x^2   \Bigl(  \left(\partial_{1} h_2(x) \right)^2 + \left(\partial_{1} h_3(x) \right)^2 +  \partial_1 h_{2}(x)  \partial_{1} h_{3}(x)\Bigr)\\
- 2 \alpha_x \left( \sum_{j=2}^n  \left(P_x\right)_{j,1}  \Bigl( 2 \partial_j h_2 (x) + \partial_j h_3 (x)\Bigr)\right) \partial_1h_2(x) \\
- 2 \alpha_x \left( \sum_{j=2}^n  \left(P_x\right)_{j,1}  \Bigl( \partial_j h_2 (x) + 2 \partial_j h_3 (x)\Bigr)\right) \partial_1h_3(x)\\
- 2 \left( \sum_{j=2}^n  \left(P_x\right)_{j,1} \partial_j h_2 (x)\right)\left( \sum_{j=2}^n  \left(P_x\right)_{j,1} \partial_j h_3 (x)\right)\\
 - 2 \left( \sum_{j=2}^n  \left(P_x\right)_{j,1} \partial_j h_2 (x)\right)^2 - 2 \left( \sum_{j=2}^n  \left(P_x\right)_{j,1} \partial_j h_3 (x)\right)^2\\
-2 \sum_{k=3}^n \left( e^{-2f_k(0)} \left(  \sum_{j=2}^n  \left(P_x\right)_{j,k} \partial_j h_2(x)\right)^2 \right) \\
-2 \sum_{k=2, k\neq 3}^n \left( e^{-2f_k(0)} \left(  \sum_{j=2}^n  \left(P_x\right)_{j,k} \partial_j h_3(x)\right)^2 \right)\\
-2  \sum_{k=4}^n  \left(  \sum_{j=2}^n  \left(P_x\right)_{j,k} \partial_j h_2(x)\right)  \left(  \sum_{j=2}^n  \left(P_x\right)_{j,k} \partial_j h_3(x)\right) \\
+ 2 e^{-2h_2(x)}  \left(  \sum_{j=2}^n  \left(P_x\right)_{j,2} \partial_j h_2(x)\right)  \left(  \sum_{j=2}^n  \left(P_x\right)_{j,2} \partial_j h_3(x)\right) \\
+ 2 e^{-2h_3(x)}  \left(  \sum_{j=2}^n  \left(P_x\right)_{j,3} \partial_j h_3(x)\right)  \left(  \sum_{j=2}^n  \left(P_x\right)_{j,3} \partial_j h_2(x)\right),
\end{multline*}
where the terms multiplying $\partial_1h_2(x)$ and $\partial_1h_3(x)$) appearing on the second and third lines, respectively, and all other remaining terms depend only on $\check{J}^2_h$ and $x$. Finally, the above formulas for $\partial_kf_2(0)$ and $\partial_kf_3(0)$ also allow to express $E_0^{\hat{g}}$ as desired. The conclusion readily follows from (\ref{ScalIdproof}).

\section{Proof of Lemma \ref{LEMToricDiscs}}\label{AppProofLEMToricDiscs}
For the proof, we introduce the following notation: for any set $S\subset \R^k$ and for any $\nu>0$, $S_{\nu}$ refers to the $\nu$-neighborhood of $S$ defined as $S_{\nu}:=S+\nu \mbox{Int}(D^k)$. \\
We prove the result by induction. For $k=2$, we define $\Sigma^2$ as the center of the disc $D^2$ and set $\tilde{\Sigma}^2:=\Sigma^2$. By construction, $\tilde{\Sigma}^2$ is a smooth  submanifold  of $\R^2$ with codimension $2$ that does not intersect $\partial D^2$. Consequently, $\tilde{\Sigma}^2$ is tranverse to $\partial D^2$, which proves (i). Next, let $\mathcal{N}$ be an open neighborhood of $\Sigma^2 \cup \partial D^2$. Then, there exists $\nu >0$ such that  $\Sigma^2_{\nu} \cup (\partial D^2)_{\nu} \subset \mathcal{N}$. The set $D^2 \setminus (\Sigma^2_{\nu/2}\cup   (\partial D^2)_{\nu/2}) $ is an annulus, diffeomorphic to $\S^1 \times [0,1]$, which proves (ii). 

Now, assume that the result holds for $k\geq 2$, and proceed to prove it for $k+1$. Set 
\[
\R^k_+ := \Bigl\{ \left(x_1, \cdots, x_k\right) \in \R^k \, \vert \, x_1 > 0\Bigr\},
\]
consider the smooth map 
\[
\begin{array}{rcl}
F \, : \, \S^1 \times \R^k_+ & \longrightarrow & \R^{k+1}\\
\left(\theta,\left(x_1, \cdots, x_k\right)\right) & \longmapsto & \left( x_1 \cos \theta, x_1 \sin \theta, x_2, \cdots, x_{k} \right),
\end{array}
\]
and define the sets $D^k_+\subset \R^k_+$ and $L^{k+1} \subset \R^{k+1}$ as follows:
\[
D_+^k := D^k \cap \R^k_+ \quad \mbox{and} \quad L^{k+1} := \Bigl\{ \left(y_1, \cdots, y_{k+1}\right) \in \R^{k+1} \, \vert \, y_1 = y_2= 0\Bigr\}.
\]
The map $F$ is a smooth diffeomorphism onto its image, and we have:
\begin{eqnarray}\label{LEMToricDiscsEQ1}
F \left( \S^1 \times D_+^k \right) = D^{k+1} \setminus L^{k+1}. 
\end{eqnarray}
By induction hypothesis, there exist sets $\Sigma^k\subset D^k$ and $\tilde{\Sigma}^k \subset \R^k$ satisfying properties (i)-(ii). Since $\tilde{\Sigma}^k$ has codimension $2$ and is transverse to the sphere $\partial D^k$, we may assume, without loss of generality, by rotating and shrinking $\tilde{\Sigma}^k$ outside $D^k$, if necessary, that there exists $\delta >0$ such that $\tilde{\Sigma}^k$ does not intersect the closed ball $\bar{B}^k(\bar{x},\delta)$ where $\bar{x}:=(-1, 0, \cdots, 0)$. For every $R>0$, let $x_R := (R, 0, \cdots, 0)$, and consider the affine diffeomorphism $\varphi_R: \R^k \rightarrow \R^k$ by $\varphi_R(x):= x_R + (R+1)x$. Using this map, define the sets:
\[
D_R^k:= \varphi_R\left( D^k\right), \quad 
\tilde{\Sigma}^k_R :=  \varphi_R\left( \tilde{\Sigma}^k\right), \quad 
\Sigma^k_R :=  \varphi_R\left( \Sigma^k\right), \quad 
\mathcal{B}_R := \varphi_R\left(\bar{B}^k(\bar{x},\delta) \right).
\]
By construction, properties (i)-(ii) are satisfied by replacing $D^k, \Sigma^k, \tilde{\Sigma}^k$ with $D^k_R, \Sigma^k_R, \tilde{\Sigma}^k_R$, the set $\mathcal{B}_R$ coincides with the closed ball $\bar{B}^k(\bar{x},(R+1)\delta)$, and $ \tilde{\Sigma}^k_R$ does not intersect $\mathcal{B}_R$.
Furthermore, there exists $\bar{R}>0$ such that for any $R \geq \bar{R}$, $D^k_R$ contains the closure of $D^k_+$, denoted $\overline{D^k_+}$, in its interior. Let $\hat{x} := (1/2, 0 \cdots, 0) \in D^k_+$, and for each $v\in \S^{k-1}$, let $\tau_+(v)$ (resp. $\tau_R(v)$ for $R\geq \bar{R}$) denote the length of the line segment formed by the intersection of the half-line $\hat{x} + [0,+\infty)v$ with $D_+^k$  (resp. $D_R^k$).
For each $R \geq \bar{R}$, the map $\tau_R :\S^{k-1} \rightarrow (0,+\infty)$ is smooth, while $\tau_+ :\S^{k-1} \rightarrow (0,+\infty)$ is Lipschitz and smooth outside the set of $v\in \S^{k-1}$ such that the half-line $\hat{x} + [0,+\infty)v$ intersects the sphere $\S^{k-1} \cap \{x_1=0\}$.
Moreover, we have $\tau_R > \tau_+$ for all $R \geq \bar{R}$. Choose a smooth, nondecreasing function $\mu: [0,+\infty) \rightarrow [0,1]$ such that $\mu(1/8)=0$ and $\mu(1/4)=1$. For every $R\geq \bar{R}$, define the map $\Psi_R : \R^k\rightarrow \R^k$ by
\[
\Psi_R(x) := \left\{
\begin{array}{ccl}
\left( 1 - \mu \left(\left| x-\hat{x}\right|\right) \right)\, x + \mu \left(\left| x-\hat{x}\right|\right)  \left[ \hat{x} + \frac{\tau_R\left( \frac{x-\hat{x}}{|x-\hat{x}|}\right)}{\tau_+\left( \frac{x-\hat{x}}{|x-\hat{x}|}\right)} \, \left( x-\hat{x} \right)\right] &\mbox{ if } & x \neq \hat{x} \\
\hat{x} & \mbox{ if } & x = \hat{x}.
\end{array}
\right.
\]
By construction, for every $R\geq \bar{R}$, $\Psi_R$ is a Lipschitz homeomorphism satisfying 
\[
\Psi_R \left( \overline{D^k_+}\right) =D^k_R.
\]
Furthermore, consider the compact set $S_R \subset D^k_R$, defined as the convex hull of the union of line segments formed by the intersection of $D_R^k$ with all half-lines of the form $\hat{x} + [0,+\infty)v$, with $v\in \S^{k-1}$, that intersect the sphere $\S^{k-1} \cap \{x_1=0\}$.
The map $\Psi_R^{-1}$ is smooth outside $S_R$ and any $y\in \partial D^k_R$ such that $\Psi_R^{-1}(y) \in \{x_1=0\}$ lies in $S_R$. Finally, note that, for every $R \geq \bar{R}$, the set $S_R$ is contained in the closed ball centered at $\bar{x}$ with radius $3$. 

To conclude the proof, we fix $R\geq \bar{R}$ such that $(R+1)\delta >3$, and define 
\[
\tilde{\Sigma}^k_+ := \Psi_R^{-1} \left( \tilde{\Sigma}^{k}_R\right).
\]
By construction, $\tilde{\Sigma}^{k}_R$ does not intersect the set $\mathcal{B}_R= \bar{B}^k(\bar{x},(R+1)\delta)$, which contains the set $S_R$, as $(R+1)\delta > 3$. Consequently, $\tilde{\Sigma}^k_+$ is an open smooth submanifold of $\R^k$ with codimension $2$. If it intersects the boundary of $D^k_+$, this occur at points in $\partial D^k \cap \R^k_+$, and the intersection is transverse. In particular, by shrinking $\tilde{\Sigma}^k$ outside $D^k$ if necessary, we may assume that $\tilde{\Sigma}^k_+ $ is contained in $\R^{k}_+$. Then, we define
\[
\Sigma^{k+1} := \tilde{\Sigma}^{k+1} \cap D^{k+1}, \quad \mbox{where} \quad \tilde{\Sigma}^{k+1} :=   L^{k+1}  \cup F\bigl( \S^1 \times \tilde{\Sigma}^k_+\bigr).
\]
We now address the verification of properties (i)-(ii) for $k+1$. The set $L^{k+1}$ is a codimension $2$ vector space that is transverse to $\partial D^{k+1}$. Additionally, $\tilde{\Sigma}^k_+ $ is an open smooth submanifold of $\R^k$ with codimension $2$, contained in $\R^{k}_+$, and transverse to $D^k_+$.
Since $F$ is a smooth diffeomorphism onto its image, the image of $\S^1 \times \tilde{\Sigma}^k_+$ under $F$ is an open smooth submanifold of $\R^{k+1}$ with codimension $2$, transverse to $\partial D^{k+1}$. Moreover, by (\ref{LEMToricDiscsEQ1}), this image does not intersect $L^{k+1}$, completing the proof of (i). To prove (ii), consider an open neighborhood $\mathcal{N}$ of $\Sigma^{k+1}\cup \partial D^{k+1}$. By the above construction, there exists an open neighborhood $\mathcal{N}^k_+$ of the closed set $\partial \overline{D^k_+}\cup (\tilde{\Sigma}^k_+ \cap D^{k})$ in $\R^k$ such that  
\[
F \bigl( \S^1 \times \left( \mathcal{N}^k_+\cap \R^k_+\right) \bigr) \subset \mathcal{N}
\]
and the set 
\[
\mathcal{N}_R^k := \Psi_R\left(\mathcal{N}^k_+\right)
\]
is an open neighborhood of $\partial D^k_R \cup \Sigma^k_R$. By the induction hypothesis, there exists a smooth map $\Phi^k_R: \T^{k-1} \times [0,1] \rightarrow \mbox{Int}(D^k_R\setminus \Sigma^k_R)$ such that $\Phi^k_R$ is a diffeomorphism onto its image and $\Phi^k_R( \T^{k-1} \times \{0,1\})\subset \mathcal{N}^k_R$. We conclude by approximating $\Psi^k_R$ with a smooth diffeomorphism via regularizing $\tau_+$. Since $\Phi^k_R$ is valued in $\mbox{Int}(D^k_R\setminus \Sigma^k_R)$ and $\Phi^k_R( \T^{k-1} \times \{0,1\})\subset \mathcal{N}^k_R$, there exists $\nu>0$ such that 
\begin{eqnarray}\label{LEMToricDiscsEQ2}
\mbox{dist} \left( \Psi_R^{-1} \left( \Phi^k_R(\alpha,t)\right), \partial D^k_+ \cup \left(\tilde{\Sigma}^k_+ \cap D^k_+\right)\right) > \nu \qquad \forall (\alpha,t) \in \T^{k-1} \times [0,1]
\end{eqnarray}
and
\begin{eqnarray}\label{LEMToricDiscsEQ3}
\Psi_R^{-1} \left( \Phi^k_R(\alpha,t)\right) \subset \left(\mathcal{N}_+^k\right)_{\nu} \qquad \forall (\alpha,t) \in \T^{k-1} \times \{0,1\}.
\end{eqnarray}
Next, consider a smooth approximation $\tau : \S^1 \rightarrow (0,+\infty)$  of $\tau_+$ in the $C^0$-topology, and define the map $\Psi : \R^k\rightarrow \R^k$ by
\[
\Psi(x) := \left\{
\begin{array}{ccl}
\left( 1 - \mu \left(\left| x-\hat{x}\right|\right) \right)\, x + \mu \left(\left| x-\hat{x}\right|\right)  \left[ \hat{x} + \frac{\tau_R\left( \frac{x-\hat{x}}{|x-\hat{x}|}\right)}{\tau\left( \frac{x-\hat{x}}{|x-\hat{x}|}\right)} \, \left( x-\hat{x} \right)\right] &\mbox{ if } & x \neq \hat{x} \\
\hat{x} & \mbox{ if } & x = \hat{x}.
\end{array}
\right.
\]
Since $\tau_R > \tau_+$ and $\Psi$ is defined radially, the map $\Psi$ is a smooth diffeomorphism if $\|\tau-\tau_+\|_{C^0}$ is sufficiently small. Furthermore, $\Psi^{-1}$ converges uniformly to $\Psi_R^{-1}$ on compact sets as $\|\tau-\tau_+\|_{C^0} \rightarrow 0$. 
Therefore, using the fact that $\Phi^k_R$ is valued in $\mbox{Int}(D^k_R\setminus \Sigma^k_R)$, along with (\ref{LEMToricDiscsEQ2})-(\ref{LEMToricDiscsEQ3}), we may assume that $\Psi$ is a smooth diffeomorphism satisfying the following properties:
\[
\Psi^{-1} \left( \Phi^k_R(\alpha,t)\right) \in  \mbox{Int}(D^k_+\setminus \left(\Sigma^k_+ \cap D_+^k\right)) \qquad \forall (\alpha,t) \in \T^{k-1} \times [0,1]
\]
and
\[
\Psi^{-1} \left( \Phi^k_R(\alpha,t)\right) \subset \mathcal{N}_+^k \qquad \forall (\alpha,t) \in \T^{k-1} \times \{0,1\}.
\]
Finally, noting that $\T^k=\S^1 \times \T^{k-1}$, we define $\Phi: \T^{k} \times [0,1]\rightarrow \mbox{Int}(D^{k+1}\setminus \Sigma^{k+1}) $ by 
\[
\Phi \left(\theta, \alpha, t \right) := F\left(\theta,\Psi^{-1}\left( \Phi^k_R (\alpha,t) \right)\right) \qquad \forall (\theta,\alpha,t )\in \S^1 \times \T^{k-1} \times [0,1].
\]
By construction, $\Phi$ is a smooth diffeomorphism onto its image, satisfying property (ii) for $k+1$.


\end{document}